\documentclass[a4paper,11pt]{article}

\usepackage{amsthm}
\usepackage{amsmath}
\usepackage{amssymb}
\usepackage{amsfonts}
\usepackage{dsfont}
\usepackage{amstext}
\usepackage{xspace}
\usepackage{authblk}
\usepackage[dvips]{graphicx}
\usepackage{url}
\usepackage{enumerate}
\usepackage[all,2cell,ps]{xy}
\usepackage[active]{srcltx} 
\newcommand{\lulab}[1]{\ar@{}[l]_<<{#1}}
\newcommand{\rulab}[1]{\ar@{}[r]^<<{#1}}
\newcommand{\ldlab}[1]{\ar@{}[l]^<<{#1}}
\newcommand{\rdlab}[1]{\ar@{}[r]_<<{#1}}
\newcommand{\edge}[1]{\ar@{-}[#1]}
\newcommand{\node}{*+[o][F-]{ }}

\theoremstyle{definition}
\newtheorem{Def}{Definition}[section]

\theoremstyle{plain}
\newtheorem{theorem}[Def]{Theorem}

\theoremstyle{plain}
\newtheorem{prop}[Def]{Proposition}

\theoremstyle{plain}
\newtheorem{lemma}[Def]{Lemma}

\theoremstyle{plain}
\newtheorem{cor}[Def]{Corollary}

\theoremstyle{remark}
\newtheorem{remark}[Def]{{\it Remark}}

\theoremstyle{remark}
\newtheorem{example}[Def]{Example}

\newcommand{\R}{\mathbb{R}}

\newcommand{\Z}{\mathbb{Z}}

\renewcommand{\P}{\ensuremath{\mathds P}}
\newcommand{\bd}[1]{\begin{Def}\label{#1}}
\newcommand{\ed}{\end{Def}}
\newcommand{\bp}[1]{\begin{prop}\label{#1}}
\newcommand{\ep}{\end{prop}}
\newcommand{\bt}[1]{\begin{theorem}\label{#1}}
\newcommand{\et}{\end{theorem}}
\newcommand{\be}[1]{\begin{equation}\label{#1}}
\newcommand{\ee}{\end{equation}}

\def\O{\Omega}
\def\s{\sigma}

\def\ra{\rightarrow}
\def\F{{\cal{F}}}
\def\o{\omega}
\def\H{{\cal{H}}}

\newcommand{\poset}{\ensuremath{{\mathcal{S}}}}

\newcommand{\beex}{\begin{example}}
\newcommand{\eex}{\end{example}}

\newsavebox{\Hasseprimo}
\begin{lrbox}{\Hasseprimo}{ %
\begin{xy}
\xymatrix{ %
&  & \node \lulab{d} & &\\
& \node \lulab{b} \edge{ur} \edge{dr} &  & \node \rulab{c} \edge{ul} \edge{dl}&\\
&  & \node \ldlab{a}\edge{d} & &\\
&  & \node \ldlab{w} & & \\
& & \ldlab{(\poset_1)} & & \\ }
\end{xy} }
\end{lrbox}

\newsavebox{\HasseDiamond}
\begin{lrbox}{\HasseDiamond}{ %
\begin{xy}
\xymatrix{ %
&  & \node \lulab{d} & &\\
& \node \lulab{b} \edge{ur} \edge{dr} &  & \node \rulab{c} \edge{ul} \edge{dl}&\\
&  & \node \ldlab{a} & &\\
& & \ldlab{\textrm{Diamond}} & & \\
}
\end{xy} }
\end{lrbox}

\newsavebox{\Hassesecondo}
\sbox{\Hassesecondo}{\xymatrix{
&  & \node \lulab{d} & & \\
& \node \lulab{b} \edge{ur} \edge{dr}  & \node \lulab{w} \edge{u}  \edge{d}  & \node \rulab{c}\edge{ul} \edge{dl} & \\
&  & \node \ldlab{a} & & \\
& & & & \\
& & \ldlab{(\poset_2)} & & \\
}}

\newsavebox{\Hassetertio}
\sbox{\Hassetertio}{\xymatrix{
&  & & \node \lulab{d}\edge{dr} \ar@{-}[ddll]& & & \\
& & & & \node \rulab{w} \edge{dr} & & \\
& \node \lulab{b} \ar@{-}[ddrr] & & & & \node \rulab{c}\ar@{-}[ddll] & \\
& & & & & & \\
& & &  \node \ldlab{a} & & & \\
& & & \ldlab{(\poset_3)}& & & \\
} }

\newsavebox{\Hassequarto}
\sbox{\Hassequarto}{\xymatrix{
& \\
& \node \lulab{c}\edge{d} \ar@{-}[ddrr] & &\node \rulab{d}\edge{dd}
\ar@{-}[ddll] &\\
& \node \lulab{w}\edge{d}\\
& \node \ldlab{a} & &\node\rdlab{b}&\\
& & \ldlab{(\poset_4)}&\\
}}

\newsavebox{\HasseBowTie}
\sbox{\HasseBowTie}{\xymatrix{
& \node \lulab{c}\ar@{-}[dd] \ar@{-}[ddrr] & &\node \rulab{d}\ar@{-}[dd]
\ar@{-}[ddll] &\\
& & & & \\
& \node \ldlab{a} & &\node\rdlab{b}&\\
 & & \ldlab{\textrm{Bowtie}} & & \\
}}

\newsavebox{\Hassequinto}
\sbox{\Hassequinto}{ \xymatrix{
& & \node \lulab{w}\edge{dr}\edge{dl}\\
& \node \lulab{c}\edge{dd} \ar@{-}[ddrr] & &\node \rulab{d}\edge{dd}
\ar@{-}[ddll] &\\
& \\
& \node \ldlab{a} & &\node\rdlab{b}&\\
& & \ldlab{(\poset_5)}&
}}

\newsavebox{\HasseDoubleBowTie}
\sbox{\HasseDoubleBowTie}{
\xymatrix{
& \node \lulab{f} \edge{d} \edge{dr} & \node \lulab{e} & \node \rulab{d} \edge{dl} \edge{d}& \\
& \node \ldlab{a} \edge{ur}  & \node \ldlab{b} \edge{u} & \node \rdlab{c} \edge{ul} & \\
& & \ldlab{\poset_8} & & \\
}}

\newsavebox{\HasseCrown}
\sbox{\HasseCrown}{\xymatrix{
& \node \lulab{f} \edge{d} \edge{dr} & \node \lulab{e} & \node \rulab{d} \edge{d}& \\
& \node \ldlab{a}  \ar@{-}[urr]  & \node \ldlab{b} \edge{u} & \node \rdlab{c} \edge{ul} & \\
& & \ldlab{\poset_7} & & \\
}}

\newsavebox{\HasseCrownVariation}
\sbox{\HasseCrownVariation}{
\xymatrix{
& \node \lulab{f} \edge{d} \edge{dr} & \node \lulab{e} & \node \rulab{d} \edge{dl} \edge{d}& \\
& \node \ldlab{a} \edge{ur} \edge{urr} & \node \ldlab{b} \edge{u} & \node \rdlab{c} \edge{ul} & \\
& & \ldlab{\poset_9} & & \\
}}

\newsavebox{\HasseCrownVariationBis}
\sbox{\HasseCrownVariationBis}{
\xymatrix{
& \node \lulab{f} \edge{d} \edge{dr} & \node \lulab{e} & \node \rulab{d} \edge{dl} \edge{d}& \\
& \node \ldlab{a} \edge{ur} \edge{urr} & \node \ldlab{b} \edge{u} & \node \rdlab{c} \edge{ul} \edge{ull} & \\
}}

\newsavebox{\HassePesce}
\sbox{\HassePesce}{\xymatrix{
& &   \node \lulab{d} & & \node \lulab{f} \edge{dl} & \\
& \node \lulab{b} \edge{ur} \edge{dr} &  & \node \ar@{}[rd]^<<{c} \edge{ul} \edge{dl}& & \\
& & \node \ldlab{a} & & \node \ldlab{e} \edge{lu} & \\
}}

\newsavebox{\HassePapillon}
\sbox{\HassePapillon}{\xymatrix{
& & \node \lulab{e} \edge{dl} \edge{dr} & & \node \lulab{f} \edge{dl} \edge{dr}& & \\
& \node \lulab{b} & & \node \node \ar@{}[rd]^<<{c} & & \node \rulab{d} & \\
& & & \node \rdlab{a} \ar@{-}[ull] \edge{u} \ar@{-}[urr] & & & \\
& & & \ldlab{\poset_6} & & & \\
}}

\newsavebox{\HasseCompleteCrown}
\sbox{\HasseCompleteCrown}{
\xymatrix{
& \node \lulab{f} \edge{dd} \edge{ddrr} & & \node \lulab{e} & & \node \rulab{d} \edge{ddll} \edge{dd}& \\
&\\
& \node \ldlab{a} \edge{uurr} \edge{uurrrr} & & \node \ldlab{b} \edge{uu} & &\node \rdlab{c} \edge{uull} \edge{uullll}& \\
}}

\newsavebox{\HassekCrown}
\sbox{\HassekCrown}{
\xymatrix{
& & \node \rulab{y_0} \ar@{-}[ddl] \ar@{-}[ddr]& &\node \rulab{y_1} \ar@{-}[ddl] \ar@{-}[ddr] %
      & & \node \rulab{y_2} \ar@{-}[ddl] \ar@{--}[dr]& & & & &\node \rulab{y_{k-1}} \ar@{-}[ddl] \ar@{-}[ddr] %
      & & \node \rulab{y_k} \ar@{-}[ddl] \ar@{-}[ddllllllllllll]& \\
& & & & & & & & &\ar@{--}[dr] & & & & & \\
& \node \rdlab{x_0} & & \node \rdlab{x_1}& & \node \rdlab{x_2}& & & & & \node \rdlab{x_{k-1}}& &\node \rdlab{x_k} & & \\
}}

\newsavebox{\Aciclico}
\sbox{\Aciclico}{
\xymatrix{
& \node \lulab{f}  \ar@{-}[dddrr] & & \node \lulab{e}  \edge{d} & & \node \rulab{d}  \ar@{-}[dll] & \\
& & &  \node \ldlab{w_2}  \edge{d} \ar@{-}[dddrr]& &  & \\
& & & \node \rulab{w} \edge{d}  & & & \\
& & & \node \rulab{w_1} \edge{d}  \edge{dll} & & & \\
& \node \rdlab{a} & & \node \rdlab{b} & & \node \rdlab{c} & \\
}}

\newsavebox{\Hasseypsilon}
\begin{lrbox}{\Hasseypsilon}{ %
\begin{xy}
\xymatrix{ %
& \node \lulab{b}  \edge{dr} &  & \node \rulab{c}  \edge{dl}&\\
&  & \node \ldlab{a}\edge{d} & &\\
&  & \node \ldlab{w} & & \\
}
\end{xy} }
\end{lrbox}

\newsavebox{\Hasseupsilon}
\begin{lrbox}{\Hasseupsilon}{ %
\begin{xy}
\xymatrix{ %
& & \node \lulab{w}  \edge{d} &  & \\
&   & \node \ldlab{a}\edge{dr} \edge{dl} & &\\
&   \node \ldlab{b}& & \node \ldlab{c} & \\
}
\end{xy} }
\end{lrbox}

\author[1]{P. Dai Pra\thanks{daipra@math.unipd.it}\thanks{P. Dai Pra and I. G. Minelli thank the ESF
research networking program RDSES and the Institut f\"{u}r Mathematik of Universit\"{a}t Potsdam for financing stays in Potsdam.}}
\author[2]{P.-Y. Louis\thanks{louis@math.uni-potsdam.de}\thanks{P.-Y. Louis thanks the Mathematics Department of the University of Padova for financing stays in Padova.}}
\author[3]{I.~G. Minelli\thanks{minelli@math.unipd.it}}
\affil[1]{Dipartimento di Matematica Pura ed Applicata,
	\mbox{Universit\`a degli Studi di Padova,}
	\mbox{via Trieste 63,}
	\mbox{I-35131 Padova,}
	Italy}
\affil[2]{Institut f\"{u}r Matematik, Universit\"{a}t Potsdam,
\mbox{Am neuen Palais, 10 -- Sans Souci}, \mbox{14469 Potsdam, Germany}}
\affil[3]{\mbox{Dipartimento di Matematica Pura ed Applicata},
	Universit\`a degli Studi di Padova,
	\mbox{via Trieste 63,}
	\mbox{I-35131 Padova,}
	Italy}

\title{Realizable monotonicity for continuous-time Markov processes}
\date{Post-print of Stochastic Processes and their Applications,\\
120 (6), June 2010, pp. 959--982 \\
doi:10.1016/j.spa.2010.03.002\\
Creative Commons Attribution Non-Commercial No Derivatives License \\
Received 20 October 2008; received in revised form 23 February 2010; accepted 26 February 2010
Available online 7 March 2010}

\begin{document}

\maketitle

\url{http://www.sciencedirect.com/science/article/pii/S0304414910000499}

\begin{abstract}
We formalize and analyze the
notions of stochastic monotonicity and realizable monotonicity for
Markov Chains in continuous-time, taking values in a finite
partially ordered set. Similarly to what happens in discrete-time,
the two notions are not equivalent. However, we show that there
are partially ordered sets for which stochastic monotonicity and
realizable monotonicity coincide in continuous-time but not in
discrete-time.
\end{abstract}

\centerline{\textbf{Keywords}}
Markov processes, coupling, partial ordering, stochastic monotonicity, realizable monotonicity, monotone random dynamical system representation
\medskip

{\textbf{2000 MSC:}
Primary: 60J27
Secondary: 60J10,
60E15, 
90C90 
}



\newpage
\section{Introduction} \label{int}
The use of Markov chains in simulation has raised a number of
questions concerning qualitative and quantitative features of
random processes, in particular in connection with mixing
properties. Among the features that are useful in the analysis of
effectiveness of Markov Chain Monte Carlo (MCMC) algorithms, a
relevant role is played by monotonicity. Two notions of
monotonicity have  been proposed, for Markov chains taking values
in a {\em partially ordered set} $S$ ({\em poset} from now on). To
avoid measurability issues, which are not relevant for our
purposes, we shall always assume $S$ to be finite. Moreover, all
Markov chains are implicitly assumed to be time-homogeneous.
\begin{Def}\label{mon1} A Markov chain $(\eta_t)$, $t \in \R^+$ or $t \in
\Z^+$, on the poset $S$, with transition probabilities $P_t(x,y)
:= P(\eta_t = y|\eta_0=x)$, is said to be \emph{stochastically monotone}
(or, more briefly, {\em monotone}) if for
each pair $w,z \in S$ with $w \leq z$ there exists a Markov chain
$(X_t(w,z))$ on $S \times S$ such that
\begin{enumerate}[i.)]
\item $X_0(w,z) = (w,z)$;
\item each component $(X^i_t(w,z))$,
$i=1,2$ is a Markov chain on $S$ with transition probabilities
$P_t(\cdot,\cdot)$;
\item $X^1_t(w,z) \leq X^2_t(w,z)$ for all $t
\geq 0$.
\end{enumerate} \end{Def}

There are various equivalent formulations of stochastic
monotonicity. For instance, defining the {\em transition operator}
$T_t f(x) := \sum_{y \in S} f(y) P_t(x,y)$, then the chain is
stochastically monotone if and only if for every~$t\geq 0$ $T_t$
maps increasing functions to increasing functions. See
Theorem~5 in~\cite{Kamae-et-al} for this generalization to
continuous-time dynamics of the well-known Strassen's result
(Theorem~2.4 in~\cite{LindvallBook}). This
characterization can be turned (see Section \ref{pre}) into a
simple algorithm for checking stochastic monotonicity  of Markov
chains in terms of the element of the transition matrix (in
discrete-time) or in terms of the infinitesimal generator (in
continuous-time).

References on the relations between  this monotonicity concept and the existence and construction of a monotone coupling for some
family of processes in continuous-time, such as diffusions or interacting particle systems, are \cite{Chen,YuPing,ForbesFrancois,LopezSanz98} and references therein.

For various purposes, including simulation, a stronger notion of
monotonicity has been introduced.
\begin{Def}\label{mon2} A Markov
chain $(\eta_t)$, $t \in \R^+$ or $t \in \Z^+$, on the poset $S$,
with transition probabilities $P_t(x,y) := P(\eta_t =
y|\eta_0=x)$, is said to be {\em realizably monotone} if there
exists a Markov chain $(\xi_t(\cdot))$ on $S^S$ such that
\begin{enumerate}[i.)]
\item $\xi_0 = \textrm{Id}$;
\item for every fixed
$z \in S$, the process $(\xi_t(z))$ is a Markov chain with
transition probabilities $P_t(\cdot,\cdot)$;
\item if $w \leq z$,
then for every $t \geq 0$ we have $\xi_t(w) \leq \xi_t(z)$.
\end{enumerate}
\end{Def}

In other words, realizable monotonicity means that we can
simultaneously couple, in an order preserving way, all processes
leaving any possible initial state. This property becomes relevant
when one aims at sampling from the stationary measure of a Markov
chain using the Propp and Wilson coupling from the past algorithm (see~\cite{PW}) which we briefly
summarize in Subsection~\ref{PW}.
Notice that if realizable monotonicity holds, the simultaneous order preserving coupling~$\xi_t$ can be extended to all $S^S$. Indeed, for $f \in S^S$, we can define $\tilde{\xi_t}:=\xi_t \circ f$. We have $\tilde{\xi}_0=f$ and property \mbox{ii.)} holds true for $\tilde{\xi_t}$, while property \mbox{iii.)} has to be replaced by $f(w) \leq f(z) \Rightarrow  \tilde{\xi}_t(w) \leq  \tilde{\xi}_t(z)$.

We recall that in \cite{FM} a more general definition of
\emph{stochastic monotonicity} and \emph{realizable monotonicity}
for a \emph{system of probability measures} is considered:
\begin{Def}\label{mon0}
Let $\mathcal{A}$ and $\mathcal{S}$ be two partially ordered sets.
A system of probability measures on $\mathcal{S}$ indexed in
$\mathcal{A}$, $(\P_\alpha: \alpha\in \mathcal{A})$
is said to be stochastically monotone if $\P_\alpha$ is stochastically smaller than $\P_\beta$ (denoted
$\P_\alpha\preccurlyeq
\P_\beta$) whenever $\alpha \leq_{\mathcal{A}} \beta$, {\it i.e.} $\int f d \P_\alpha \leq_{\mathbb{R}}
\int f d \P_\beta $ for every increasing function $f:\mathcal{S}\rightarrow
\mathbb{R}$
whenever $\alpha\leq_{\mathcal{A}}\beta$.\\
The system $(\P_\alpha: \alpha\in \mathcal{A})$ is said to be
realizably monotone if there exists a probability space $(\Omega,
\mathcal{F}, \mathbb{P})$ and a system of $\mathcal{S}$-valued
random variables $(X_\alpha: \alpha\in \mathcal{A})$ with $\P (X_\alpha \in \cdot)= \P_\alpha (\cdot)$
such that
$X_\alpha\leq_{\mathcal{S}} X_\beta$ whenever $\alpha\leq_{\mathcal{A}}\beta$.\\
\end{Def}

If the transition probabilities, or the infinitesimal generator,
are given, no simple rule for checking realizable monotonicity is
known. Since, obviously, realizable monotonicity implies
stochastic monotonicity, a natural question is to determine for
which posets the converse is true. This problem has been
completely solved in~\cite{FM} for discrete-time Markov chain (see
Theorem~\ref{fill-machida} here). More precisely, it has been
solved as a particular case of the more general problem of
\emph{equivalence} between stochastic and realizable monotonicity
for systems of probability measures
$\P_{\alpha}$ w.r.t $(\mathcal{A},\mathcal{S})$
where $\mathcal{A}=\mathcal{S}=S$, $\P_\alpha=P(\alpha, \ \cdot \ )$ with $P(\cdot,\cdot)$ denoting the Markov chain's transition probability on $S$.
Notice that the realizable monotonicity of $(P(\alpha, \ \cdot \ ),\ \alpha \in S)$ is equivalent to the one for discrete-time Markov chains given in~Definition~\ref{mon2}
through the construction
$$ \left\lbrace \begin{array}{l}
\xi_0(z)=z, \\
\xi_t(z)=X^{(t)}_{\xi_{t-1}(z)} \ \forall t \geq 1 \\
\end{array} \right. $$
where $(X^{(t)}_\alpha: \alpha\in S)$ ($t \in \mathbb Z^+$) is an i.i.d. sequence of copies of $(X_\alpha: \alpha\in S)$, which realizes $(P(\alpha, \ \cdot \ ),\ \alpha \in S)$.

In what follows, when such
equivalence holds we shall say that
\emph{monotonicity equivalence holds}.
\medskip

Let us now give the following definitions
\bd{HasseDiagramDef}
The \emph{Hasse diagram}
of a poset is an oriented graph. Its vertices are the elements of the poset. There is an edge from $x$ to $y$ if $x \preceq y$ and
$x \preceq z \preceq y$ implies $z=x$ or $z=y$  ; it is said that $y$ \emph{covers} $x$. \\

The \emph{cover graph} of a poset $S$ is the Hasse diagram regarded as an undirected graph.
\ed

By convention, $y$ is drawn above $x$ in the planar representation of the diagram in order to mean there is an edge from $x$ to $y$. With this convention of reading the diagram from bottom to the top there is no need to direct any edges. See for example figures~\ref{4posets}, \ref{5posets}, \ref{6posets}.
\medskip

In the case of discrete-time Markov chains the following result holds. It is a consequence of {Theorem 4.3} stated in  \cite{FM},
(see the previous comment).

\newpage
\begin{theorem}\label{fill-machida}

Every stochastically monotone Markov chain in
the poset $S$ is also realizably monotone if and only if the cover
graph of $S$ is {\em acyclic}, {\it i.e.} there is no loop $x_0,\
x_1, \ldots, x_n,\ x_{n+1} = x_0$ such that, for $i=0,1,\ldots,n$,
\begin{enumerate}[i.)]
\item
$x_i \neq x_{i+1}$;
\item either $x_i$ covers $x_{i+1}$ or $x_{i+1}$ covers $x_i$.
\end{enumerate}
\end{theorem}

In the following we call acyclic a poset which has an acyclic cover graph.
The nontrivial proof of the above statement consists
of three steps.
\begin{enumerate}
\item For each {\em minimal} cyclic poset an example is found of a
stochastically monotone Markov chain which is not realizably
monotone. \item Given a general cyclic poset, a stochastically
monotone but not realizably monotone Markov chain is constructed
by ``lifting'' one of the examples in step 1. \item A proof by
induction on the cardinality of the poset shows that, in an
acyclic poset, stochastically monotone Markov chains are
realizably monotone.
\end{enumerate}
Note there is no contradiction with the fact that on some cyclic posets, such as product spaces, order preserving coupling may exist for \emph{some} monotone Markovian dynamics. See for instance~\cite{LouisCoupling}.
\medskip


Our aim in
this paper is to deal with monotonicity equivalence in continuous-time
for time-homogeneous {\em regular} Markov chains, {\it i.e.} Markov chains possessing an
infinitesimal generator (or, equivalently, jumping a.s. finitely
many times in any bounded time interval).
It turns out that if in a poset $S$ stochastic monotonicity
implies realizable monotonicity in discrete-time, then the same
holds true in continuous-time (see Corollary~\ref{impl}). The
converse is not true, however; for the posets whose Hasse diagram
is represented in Figure~\ref{4posets} (the diamond and the
bowtie, following the terminology in~\cite{FM}) equivalence
between stochastic monotonicity and realizable monotonicity holds
in continuous-time but not in discrete-time.

\begin{figure}[h]
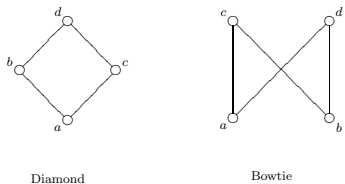

\centerline{ \scalebox{0.6}{
\usebox{\HasseDiamond} %
\usebox{\HasseBowTie} %
}}
\caption{\label{4posets}The four-points posets for which there is no equivalence between the two notions of monotonicity in discrete time.}
\end{figure}

In this paper we do not achieve the goal of characterizing all
posets for which equivalence  holds in continuous time. Via a computer-assisted (but
exact) method we find a complete list of five- and six-point posets
for which equivalence fails (see Propositions~\ref{Prop3.1} and~\ref{Prop3.2}).
Moreover we state in the Proposition~\ref{extension} the following: in each poset
containing one of the former as an induced sub-poset (see Definition~\ref{sub-poset}),
equivalence fails as
well (this does not follow in a trivial way).\\

In Section~\ref{pre} we give some preliminary notions, whose aim
is to put the realizable monotonicity problem in continuous-time
on a firm basis. In Section~\ref{equ} we perform a systematic
investigation of the monotonicity equivalence for five and six
points posets, using the software \emph{cdd+} (see \cite{cdd}).
Extensions to larger posets are presented in Section~\ref{ext}.
Some further considerations and conjectures are contained in
Section~\ref{conclusions}.

\section{Preliminaries} \label{pre}

\subsection{Characterization of realizable monotonicity}
\label{comp} Let $(S,\leq)$ be a finite poset, and $L =
(L_{x,y})_{x,y \in S}$ be the infinitesimal generator of a regular
continuous-time Markov chain on $S$. Assume the chain is
realizably monotone and let $(\xi_t(\cdot))_{t \geq 0}$ be an
order-preserving coupling. Since the original Markov chain is
regular, also $(\xi_t(\cdot))_{t \geq 0}$ must be regular: if not,
for at least one $z \in S$, $\xi_t(z)$ would be not regular, which
is not possible by condition {ii.} in Definition \ref{mon2}. Thus
$(\xi_t(\cdot))_{t \geq 0}$ admits an infinitesimal generator
${\cal{L}} = ({\cal{L}}_{f,g})_{f,g \in S^S}$.  Let now be
$\varphi: S \rightarrow \R$, $z \in S$, and define
$F_{\varphi,z}:S^S \rightarrow \R$ by $F_{\varphi,z}(f) :=
\varphi(f(z))$. The fact that each component of the chain
generated by ${\cal{L}}$ is a Markov chain with generator $L$ is
equivalent to the following statement: for all choices of
$\varphi$, $z$, and all $f \in S^S$,
\begin{equation}\label{2.1} {\cal{L}} F_{\varphi,z}(f) = L \varphi(f(z)). \end{equation} By an
elementary algebraic manipulation of (\ref{2.1}), we can
re-express (\ref{2.1}) with the following statement: for every
$z,x,y \in S$, $x \neq y$ and every $f \in S^S$ such that $f(z)
=x$, we have \begin{equation}\label{2.2} L_{x,y} = \sum_{g \in S^S
: g(z) = y} {\cal{L}}_{f,g}. \end{equation} Now let $\textrm{Id}$
denote the identity on $S$, and define $\Lambda(f) :=
{\cal{L}}_{\textrm{Id},f}$. Note that since the Markov chain
generated by ${\cal{L}}$ is order preserving, necessarily
${\cal{L}}_{\textrm{Id},f}>0 \ \Rightarrow \ f \in {\cal{M}}$,
where ${\cal{M}}$ is the set of increasing functions from $S$ to
$S$. Note that, by (\ref{2.2}), for $x \neq y$,
\begin{equation}\label{2.3} L_{x,y} = \sum_{f \in
{\cal{M}}:f(x)=y} \Lambda(f). \end{equation} Identity (\ref{2.3})
characterizes the generators of realizably monotone Markov chains,
in the sense specified by the following proposition.
\begin{prop}\label{prop2.1} A generator $L$ is the generator of a
realizably monotone Markov chain if and only if there exists
$\Lambda: {\cal{M}} \rightarrow \R^+$ such that (\ref{2.3}) holds.
\end{prop}
\begin{proof} One direction has been proved above. For the converse,
suppose (\ref{2.3}) holds for some $\Lambda: {\cal{M}} \rightarrow
\R^+$. For $f,g \in S^S$, define
\[
{\cal{L}}_{f,g} := \sum_{h \in {\cal{M}}: g=h \circ f} \Lambda(h),
\]
where the sum over the empty set is meant to be zero. It is easily
checked that the Markov chain generated by ${\cal{L}}$ is order
preserving. Moreover, using (\ref{2.3}), a simple computation
shows that (\ref{2.2}) holds, completing the proof. \end{proof}

\subsection{The cones of stochastically monotone and realizably monotone generators}
\label{cones} Let $S_2 := S \times S \setminus \{(x,x): x \in S \}$.
An infinitesimal generator is a matrix $L = (L_{x,y})_{x,y \in S}$
whose non-diagonal elements are nonnegative, while the terms in
the diagonal are given by $L_{x,x} = -\sum_{y \neq x} L_{x,y}$.
Thus $L$ may be identified with an element of the cone
$(\R^+)^{S_2}$. \begin{def}\label{upset} A subset $\Gamma
\subseteq S$ is said to be an {\em up-set} (or \emph{increasing} set) if
\[
x \in \Gamma \mbox{ and } x \leq y \ \Rightarrow \ y \in \Gamma.
\]
\end{def} The following proposition (see e.g.~\cite{Massey87} for the proof) gives a
characterization of the generators of stochastically monotone Markov chains.
\begin{prop}\label{prop2.2} An element $L \in (\R^+)^{S_2}$ is the generator of a
stochastically monotone Markov chain if and only if for every up-set $\Gamma$ the
following conditions hold:
\[ \bullet \
x \leq y \not\in \Gamma \ \Rightarrow \ \sum_{z \in \Gamma}
L_{x,z} \leq \sum_{z \in \Gamma} L_{y,z};
\]
\[\bullet \
x \geq y \in \Gamma \ \Rightarrow \ \sum_{z \not\in \Gamma}
L_{x,z} \leq \sum_{z \not\in \Gamma} L_{y,z}.
\]
\end{prop}

\begin{remark}\label{remark1}
In what follows we shall often call \emph{monotone generator} (respectively \emph{realizably monotone generator}) the
infinitesimal generator of a stochastically monotone Markov Chain (resp. realizably monotone Markov Chain). Given a
generator $L$ on $S$, $x\in S$ and  $\Gamma\subset S$, we shall
use the symbol $L_{x, \Gamma}$ to denote $\sum_{z \in \Gamma}
L_{x,z}$ . Moreover, in order to check stochastic monotonicity of a generator
we shall use the following condition, which is equivalent to the
one given
in Proposition~\ref{prop2.2}:
\begin{enumerate}[i)]
\item $ \mbox{for every up-set }\Gamma,\ \ \  x \leq y \notin \Gamma \
\Rightarrow \ L_{x,\Gamma}\leq L_{y,\Gamma};
$
\item $ \mbox{for every down-set }\Gamma,\ \ \ y \geq x \notin \Gamma
\ \Rightarrow \ L_{x,\Gamma} \geq L_{y,\Gamma}$
\end{enumerate}
where a down-set
is a subset $\Gamma\subset S$ such that $ x \in \Gamma \mbox{ and
} y \leq x \ \Rightarrow \ y \in \Gamma.$

\end{remark}

Let $V := \R^{S_2}$ be provided with the natural Euclidean scalar
product $\langle \cdot, \cdot \rangle$. For given $x,y \in S$,
$\Gamma$ up-set, let $W^{\Gamma,x,y} \in \R^{S_2}$ be defined by
\[
W^{\Gamma,x,y}_{v,z} = \left\{ \begin{array}{ll}
1 & \mbox{for } \left\lbrace \begin{array}{l} x \leq y \not\in \Gamma, \, v = y, \, z \in \Gamma  \\
                       \mbox{or } x \geq y \in \Gamma, \, v=y, \, z \not\in \Gamma ;
                       \end{array} \right.
\\
-1 & \mbox{for } \left\lbrace\begin{array}{l} x \leq y \not\in \Gamma, \, v = x, \, z \in \Gamma \\
                        \mbox{or } x \geq y \in \Gamma, \, v=x, \, z \not\in
                       \Gamma ;
                       \end{array} \right.
\\
0 & \mbox{in all other cases.}
\end{array}
\right.
\]
Proposition \ref{prop2.2} can be restated as follows: $L \in
(\R^+)^{S_2}$ generates a stochastically monotone Markov chain if and only if
\be{Cone_mon_H-representation}
\langle L, W^{\Gamma,x,y} \rangle \geq 0 \ \ \mbox{for every }
\Gamma,x,y.
\ee
In other words, denoting by ${\cal{G}}_{mon}$ the set of monotone
generators, the elements of ${\cal{G}}_{mon}$ are characterized by the inequalities
\[
\begin{array}{rcl}
\langle L, W^{\Gamma,x,y} \rangle &  \geq  & 0 \\
L_{x,y} & \geq & 0
\end{array}
\]
for every $\Gamma,x,y$. In other words we are giving
${\cal{G}}_{mon}$ through the {\em rays} of its {\em polar cone}
(see~\cite{Ziegler}), {\it i.e.} the family of vectors $\{W^{\Gamma ,x,y},
\delta^{x,y}:(x,y)\in S_2,\ ,\Gamma \mbox{ up-set in } S\}$, where
\[
\delta^{x,y}_{v,z} = \left\{ \begin{array}{ll} 1 & \mbox{if }
(v,z) = (x,y) \\ 0 & \mbox{otherwise.} \end{array} \right.
\]

Proposition \ref{prop2.1} can also be restated to characterize the
set of generators of realizably monotone Markov chains as a cone
in $V$. For $f \in {\cal{M}}$, let $\mathbb{I}_f \in (\R^+)^{S_2}$
be defined by
\be{identity}
(\mathbb{I}_f)_{x,y}  = \left\{ \begin{array}{ll} 1 & \mbox{if }
f(x)=y \\ 0 & \mbox{otherwise.} \end{array}\right.
\ee
Then the set ${\cal{G}}_{r.mon}$ of generators of realizably
monotone Markov chains is the cone given by linear combination
with nonnegative coefficients of the vectors $\mathbb{I}_f $, {\it i.e.}
\be{Cone_compl_mon_V-representation}
L = \sum_{f \in {\cal{M}}} \Lambda_f \mathbb{I}_f,
\ee
with $\Lambda_f \geq 0$. Note that in this case, since for each $f \in {\cal{M}}$,
$\Gamma$ up-set, $x,y \in S$, we have
\mbox{$ \langle \mathbb{I}_f,W^{\Gamma,x,y} \rangle \geq 0,$}
we recover the inclusion $\mathcal{G}_{r.mon}\subseteq
\mathcal{G}_{mon}$. Our aim is to determine for which posets the
converse inclusion holds true.\\

\noindent In the next sections, when we want to emphasize the dependence on $S$,
we shall use the notations ${\cal{G}}_{mon}(S)$ and  ${\cal{G}}_{r.mon}(S)$.

\subsection{Comparison with the discrete-time case} \label{com}
In this subsection we establish a comparison with the
discrete-time case. The claim of  Corollary~\ref{impl} below
relies on analogous representations in terms of cones for
discrete-time transition matrices.\\

Consider a Markov chain with transition matrix $P$. We recall the
following fact.

\begin{prop}\label{mond} $P = (P_{x,y})_{x,y \in S}$ is the transition matrix
of a stochastically monotone Markov Chain if and only if, for every up-set
$\Gamma$, the map $x \mapsto \sum_{y \in \Gamma} P_{x,y}$ is
increasing. \end{prop}

This Proposition~\ref{mond} derives
from the following general
statement, which is an immediate consequence of the Definition~\ref{mon0} 
: a system of
probability measures $(P_\alpha : \alpha\in \mathcal{A})$ on
$\mathcal{S}$ is stochastically monotone if and only if, for every
up-set $\Gamma\subset \mathcal{S}$ the map $\alpha \mapsto P_\alpha
(\Gamma)$ is increasing.

The discrete-time version of the argument in Subsection \ref{comp}
shows that $P$ is the transition matrix of a realizably monotone
Markov chain if and only if there exists a probability $\Pi$ on
${\cal{M}}$ such that \begin{equation}\label{realdisc} P = \sum_{f
\in {\cal{M}}} \Pi_f \mathbb{I}_f, \end{equation} where, with a
slight abuse of notation, $\mathbb{I}_f$ given by (\ref{identity}) is now seen as a square
matrix, with the diagonal terms too.
A transition probability $P$ may be seen as an element of the convex set $[0,1]^{S_2}$. Note in this representation the identity matrix $I$ is the origin.
Analogously to the continuous-time case, the set of
stochastically monotone transition probabilities~${\mathcal K}_{mon}$ (resp. realizable monotone~${\mathcal K}_{r.mon}$) is a convex set.

In what follows, we make use of the following simple fact: if $L$ is a stochastically monotone generator, with any $\epsilon>0$ such that $\epsilon \leq (\max_x \sum_{z \neq x} L_{x,z})^{-1}$,
$P:=I+\epsilon L$ is a stochastically monotone transition probability. Conversely, for all $\epsilon>0$, $L:=\frac{1}{\epsilon} (P-I)$ is a stochastically monotone generator whenever $P$ is a  stochastically monotone transition probability.

\begin{Def} \label{wme}
The  \emph{weak monotonicity-equivalence} holds
for a poset $S$ if we can find a realizably monotone transition
matrix $(1-\lambda)I + \lambda P$ for some $\lambda \in ]0,1]$ whenever
$P$ is stochastically monotone.
\end{Def}
\begin{remark} \label{rem-wm}
By the correspondence between stochastically monotone generators and transitions probabilities, Definition~\ref{wme} is equivalent to the following statement: for any stochastically monotone generator $L$ on
$S$ there exists $\epsilon \in ]0,1]$ such that $I+\epsilon L$ is a
realizably monotone transition probability on $S$.
\end{remark}

\begin{prop}\label{weak-equiv}
Monotonicity equivalence holds for continuous-time Markov
chain on $S$ if and only if weak monotonicity equivalence
holds for discrete-time Markov chains on $S$.
\end{prop}

\begin{proof}
In view of Remark~\ref{rem-wm}, we have to show that the following two statement are equivalent for a generator $L$:
\begin{itemize}
\item[i)]
$L$ is a realizably monotone generator;
\item[ii)]
there exists $\epsilon>0$ such that $I+ \epsilon L$ is a realizably monotone transition probability.
\end{itemize}
We first show $ii) \, \Rightarrow \, i)$.
 Suppose that there exists $\epsilon>0$ such that the transition
probability $I+\epsilon L=\widetilde{P}$ is realizably monotone, which means
\[
\widetilde{P} = \sum_{f
\in {\cal{M}}} \Pi_f \mathbb{I}_f
\]
for a suitable probability $\Pi$ on ${\cal{M}}$. Thus the following identity holds in $(\R^+)^{S_2}$:
\[
L = \frac{1}{\epsilon} \sum_{f
\in {\cal{M}}} \Pi_f \mathbb{I}_f,
\]
which implies that $L$ is realizably monotone.
\\
We now show $i) \, \Rightarrow \, ii)$.
Let $L$ be a realizably monotone generator. Then, for
$\epsilon>0$ sufficiently small, $\widetilde{P}=I+\epsilon L$ is a stochastically
monotone transition probability. Let us show that
$\widetilde{P}$ is realizably monotone for $\epsilon >0$ small enough. We have the following representation of $L$ as an
element of $(\R^+)^{S_2}$: $L=\sum_{f \in {\cal{M}}} \Lambda_f \mathbb{I}_f$. Note that $f = \textrm{Id}$ gives no contribution to the sum above, since $ \mathbb{I}_{\textrm{Id}}$ is the zero element of $(\R^+)^{S_2}$. We may therefore assume $ \Lambda_{\textrm{Id}}=0$.
Now, consider the matrix in $\mathbb R^{S\times S}$ given by
$$H=\epsilon \sum_{f \in {\cal{M}}} \Lambda_f {\mathbb{I}}_f.$$
We have
$$
H_{x,y}  = \left\{
\begin{array}{ll} \epsilon\sum_{f \in {\cal{M}}} \Lambda_f\delta_{\{f(x),y\}}=
\epsilon L_{x,y}  & \mbox{if } x\neq y \\
\epsilon\sum_{f \in {\cal{M}}}
\Lambda_f\delta_{\{f(x),x\}}=\epsilon L_{x,x}+\epsilon\sum_{f \in {\cal{M}}}\Lambda_{f} & \mbox{otherwise.}
\end{array}\right.
$$
where $\delta_{\{x,y\}}$ denotes the Kronecker delta.\\
Indeed, we have
\begin{eqnarray*} \epsilon L_{x,x}&=&-
\sum_{y:y\neq x}\epsilon L_{x,y}=-\epsilon \sum_{y:y\neq x}\sum_{f \in {\cal{M}}} %
\Lambda_f\delta_{\{f(x),y\}} \\
& =& -\epsilon\sum_{f \in {\cal{M}}}\sum_{y:y\neq x} %
\Lambda_f\delta_{\{f(x),y\}}=-\epsilon\sum_{f \in {\cal{M}}}\Lambda_f(1-\delta_{\{f(x),x\}})\\
&=&\epsilon\sum_{f \in {\cal{M}}} \Lambda_f\delta_{\{f(x),x\}}-\epsilon \sum_{f \in {\cal{M}}}\Lambda_f.
\end{eqnarray*}
Therefore,
we have
$$
\widetilde{P}=I+\epsilon L=I+\epsilon \sum_{f \in {\cal{M}}}
\Lambda_f {\mathbb{I}}_f-(\epsilon\sum_{f \in {\cal{M}}}
\Lambda_f)I=\sum_{f \in {\cal{M}}}\epsilon\Lambda_f
{\mathbb{I}}_f+ (1-\sum_{f \in
{\cal{M}}}\epsilon\Lambda_f){\mathbb{I}}_{\textrm{Id}} .$$ If we
choose $\epsilon$ sufficiently small, we can interpret the
quantities
\[
\Pi_f := \left\{ \begin{array}{ll} \epsilon \Lambda_{f} & \mbox{for } f \in {\cal{M}}, f \neq \textrm{Id}\\
1-\sum_{f \in {\cal{M}}}\epsilon\Lambda_f & \mbox{for } f = \textrm{Id}
\end{array} \right.
\]
as probabilities on ${\cal{M}}$ and so
$\widetilde{P}$ is realizably monotone.\\
\end{proof}
The following fact is an immediate consequence of Proposition \ref{weak-equiv}.
\begin{cor}\label{impl}
Suppose that in the poset $S$ stochastic monotonicity and realizable monotonicity are equivalent notions for
discrete-time Markov chains. Then the equivalence holds for $S$-valued continuous-time Markov Chains as well.
\end{cor}

We summarize, in the following scheme, the main facts relating continuous and discrete-time. All matrices are thought as elements of  $\mathbb R^{S_2}$, so the diagonal is not considered.
\begin{itemize}
\item $P \in \mathcal K_{mon} \Rightarrow \forall \epsilon>0, \ L:=\frac{1}{\epsilon}P \in \mathcal G_{mon}$;
\item $P \in \mathcal K_{r.mon} \Rightarrow \forall \epsilon>0, \ L:=\frac{1}{\epsilon}P \in \mathcal G_{r.mon}$;
\item $L \in \mathcal G_{mon} \Rightarrow \exists \epsilon_L,\ \forall \epsilon \leq \epsilon_L,\ P:=\epsilon L \in  \mathcal K_{mon}$;
\item $L \in \mathcal G_{r.mon} \Rightarrow \exists \epsilon_0,\ P:=\epsilon_0 L \in  \mathcal K_{r.mon}$;
\item the weak monotonicity equivalence means the segment $[I,P]$  intersects
$\mathcal K_{r.mon}$ whenever $P \in  \mathcal K_{mon}$.
\end{itemize}

\subsection{The ''coupling from the past'' algorithm revisited}
\label{PW}

It is well known (see e.g. \cite{Arnold, Bremaud, FG, Ida}) that regular finite state Markov processes can be realized as {\em Random Dynamical Systems with independent increments} (shortly RDSI). To set up notations, let $(\O,\F,\P)$ be a probability space, and $(\theta_t)_{t \in \R}$ be a one-parameter group ({\it i.e.} $\theta_{t+s} = \theta_t \circ \theta_s$, $\theta_0 = \textrm{Id} = \mbox{ identity map}$) of $\P$-preserving maps from $\O$ to $\O$, such that the map $(t,\o) \mapsto \theta_t \o$ is jointly measurable in $t$ and $\o$. We still denote by $S$ a finite set, representing the state space.
\bd{RDS}
A {\em Random dynamical system} is a measurable map $\varphi: \R^+ \times \O \ra S^S$ such that
\begin{eqnarray}
\label{defRDS}
\varphi(0,\o) &  \equiv & \textrm{Id} \\ \varphi(t+s,\o) & = & \varphi(t,\theta_s \o) \circ \varphi(s,\o)
\end{eqnarray}
for every $s,t \geq 0$ and $\o \in \O$. \ed Note that, for $t$
fixed, $\varphi(t,\cdot)$ can be seen as a $S^S$-valued random
variable. \bd{RDSI} A Random Dynamical System $\varphi$ is said to
have {\em independent increments}  if for each $0 \leq t_0 < t_1 <
\cdots < t_n$ the random variables
\[
\varphi(t_1 - t_0, \theta_{t_0} \cdot),\ \varphi(t_2 - t_1,
\theta_{t_1} \cdot), \ \ldots, \ \varphi(t_n - t_{n-1},
\theta_{t_{n-1}} \cdot)
\]
are independent.
\ed
In what follows we consider the $\s$-fields
\[
\F_t := \s \{ \varphi(s,\cdot) : \ 0 \leq s \leq t\}
\]
\[
\F^- := \s \{\varphi(s,\theta_{-t} \cdot) : \ 0 \leq s \leq t,\ 0 \leq t \leq +\infty  \}.
\]

The following proposition recalls Theorem 1.2.1 of \cite{Dubischar} (see section~2.1.3 in \cite{Arnold} too).
\bp{MarkovRDS}
For a RDSI,
the random process $(\varphi(t,\cdot))_{t \geq 0}$ is a
$S^S$-valued, time-homogeneous, $\F_t$-Markov process. Moreover,
for any $S$-valued, $\F^-$-measurable random variable $X$ (in
particular any $X=x \in S$ constant), the process
$(\varphi(t,\cdot)(X))_{t \geq 0}$ is a $S$-valued,
time-homogeneous Markov process, whose transition probabilities do
not depend on $X$. \ep

The processes $(\varphi(t,\cdot)(x))_{t \geq 0}$, with $x \in S$
are called the {\em one point motions} of $\varphi$. When the one
point motions are Markov process with infinitesimal generator $L$,
we say that $\varphi$ {\em realizes} $L$. It is nothing else than
a complete coupling: copies of the chain starting from every
initial conditions are realized on the same probability space.

If we are given a generator $L$ of a Markov chain, it is not difficult to realize it by a RDSI. Let $\H$ be the set whose elements are the locally finite subsets of $\R$. An element $\eta \in \H$ can be identified with the $\sigma$-finite point measure $\sum_{t \in \eta} \delta_t$; the topology in $\H$ is the one induced by vague convergence, and the associated Borel $\s$-field provides a measurable structure on $\H$. Set $\O' := \H^{S_2}$. For $\o = (\o_{xy})_{(x,y) \in S_2} \in \O'$, we define
\[
\theta_t \o := (\theta_t \o_{xy})_{(x,y) \in S_2},
\]
where $\tau \in \theta_t \o_{xy} \iff \tau - t \in \o_{xy}$.  Consider now a probability $\P$ on $\O'$ with the following properties:
\begin{enumerate}[i.)]
\item
for $(x,y) \in S_2$, $\o_{xy}$ is, under $\P$, a Poisson process of intensity $L_{x,y}$;

\item
for $x \in S$ fixed, the point processes $(\o_{xy})_{y \neq x}$ are independent under $\P$;

\item
for every $I,J$ {\em disjoint} intervals in $\mathbb R$, the two families of random variables
\[
\{|\o_{xy} \cap I|: (x,y) \in S_2\} \ \mbox{ and } \ \{|\o_{xy} \cap J|: (x,y) \in S_2\}
\]
are independent under $\P$;

\item
for every $t \in \R$, $\P$ is $\theta_t$-invariant.
\end{enumerate}
It is easy to exhibit one example of a $\P$ satisfying i.-iv.: if  $\P_{xy}$ is the law of  a Poisson process of intensity $L_{x,y}$, then we can let $\P$ be the product measure
\be{prod}
\P := \otimes_{(x,y) \in S_2} \P_{xy}.
\ee
We now construct the map $\varphi$ {\em pointwise} in $\o$. Define
\[
\O = \{ \o \in \O' : \o_{xy} \cap \o_{xz} = \emptyset \ \mbox{ for every } (x,y),(x,z) \in S_2, \ y \neq z\}.
\]
By condition ii. on $\P$, $\P(\O) = 1$, and clearly $\theta_t \O = \O$ for every $t \in \R$. For every $\o \in \O$ the following construction is well posed:
\begin{itemize}
\item
set $\varphi(0,\o) = \textrm{Id}$ for every $\o \in \O$;
\item
we run the time in the forward direction. Whenever we meet $t \in \bigcup_{(x,y) \in S_2} \o_{xy}$ we use the following updating rule:
\[
\mbox{if } \varphi(t^-,\o) (x) = z \ \mbox{ and} \ t \in \o_{zy} \ \mbox{ then } \ \varphi(t,\o)(x) := y.
\]
\end{itemize}
\bp{propMarkovRDSI} (see Theorem 3.1 of \cite{Ida}). The map
$\varphi$ constructed above is a RDSI, and its one-point motions
are Markov chains with generator $L$. \ep

Conditions i.- iii. leave a lot of freedom on the choice of $\P$. The choice corresponding to (\ref{prod}) is the simplest, but may be quite inefficient when used for simulations.

The following Theorem is just a version of the ''coupling from the past'' algorithm for perfect simulation (\cite{PW}).
\bt{ProppWilson}
Let $L$ be the generator of an irreducible Markov chain on $S$, $\pi$ be its stationary distribution, and let $\varphi$ be a RDSI whose one-point motions are Markov chains with generator $L$. Define
\[
T(\o) := \inf \{ t>0 : \ \varphi(t, \theta_{-t}\o) =
\mbox{constant}\}
\]
where, by convention, $\inf \emptyset := +\infty$. Assume
$T<+\infty$ $\P$-almost surely. Then for each $x \in S$ the random
variable $\varphi(T,\theta_{-T} \cdot)(x)$ has distribution $\pi$.
\et
\begin{proof}
Set $X(\o) := \varphi(T,\theta_{-T} \o)(x)$ that, by definition of $T$, is independent of $x \in S$. For $h>0$
\[
\varphi(T+h, \theta_{-T-h} \o)(x) = \varphi(T,\theta_{-T} \o)(\varphi(h, \theta_{-T-h} \o)(x)) = X(\o).
\]
Thus, since we are assuming $T<+\infty$ a.s., we have
\[
X(\o) = \lim_{t \ra +\infty} \varphi(t,\theta_{-t} \o)(x).
\]
In particular, this last formula shows that $X$ is $\F^-$-measurable. Denote by $\rho$ the distribution of $X$, {\it i.e.} $\rho(x) := P(X=x)$. We have:
\begin{multline*}
\varphi(h,\o)(X(\o)) = \lim_{t \ra +\infty}\varphi(h,\o)( \varphi(t,\theta_{-t} \o)(x)) = \lim_{t \ra +\infty} \varphi(t+h,\theta_{-t} \o)(x) \\ =   \lim_{t \ra +\infty} \varphi(t+h,\theta_{-t-h} (\theta_h \o))(x) = X(\theta_h \o).
\end{multline*}
By Proposition \ref{MarkovRDS}, $\varphi(h,\cdot)(X(\cdot))$ has
distribution $\rho e^{hL}$. But, since $\theta_h$ is
$\P$-preserving, $X(\theta_h \cdot)$ has distribution $\rho$. Thus
$\rho = \rho e^{hL}$, {\it i.e.} $\rho$ is stationary, and
therefore $\rho = \pi$.

\end{proof}

The condition $\P(T < +\infty) =1$ depends on the particular
choice of the RDSI, and it is not granted by the irreducibility of
$L$. For example, consider $S = \{0,1\}$ and $L$ given by $L_{0,1}
= L_{1,0} =1$. We can realize this chain by letting $\o_{01}$ be a
Poisson process of intensity $1$, and $\o_{10} = \o_{01}$. Clearly
$\varphi(t,\theta_{-t} \o)(0) \neq \varphi(t,\theta_{-t} \o)(1)$
for every $t>0$, so $T \equiv +\infty$. On the other hand, and
this holds in general, if we make the choice of $\P$ given by
(\ref{prod}), it is not hard to see that $(\varphi(t,\cdot))_{t
\geq 0}$ is an irreducible Markov chain on $S^S$. By recurrence,
any constant function is hit with probability $1$ in finite time,
so $T < +\infty$ a.s.

In order to implement the Propp-Wilson algorithm, in principle one needs to run a Markov chain on $S^S$, which may be computationally unachievable. Some additional structure can make the algorithm much more effective.
\bd{monRDS}
Let $S$ be a poset. A RDS on $S$ is said to be {\em monotone} if for every $t \geq 0$ and $\o \in Q$
\[
\varphi(t,\o) \in {\cal{M}}.
\]
\ed
Suppose a Markov chain is realized by a monotone RDSI. Then
\[
\varphi(t, \theta_{-t} \o) \mbox{ is constant } \ \iff \ \varphi(t, \theta_{-t} \o) \mbox{ is constant on $A$},
\]
where $A$ is the set of points in $S$ that are either maximal or minimal. Thus to implement the Propp-Wilson algorithm one needs to run Markov chains starting from every point of $A$, that may be much smaller than $S$. Moreover, the following result holds.
\bt{monPW}
Let $S$ be a {\em connected} poset ({\it i.e.} every two points in $S$ are connected by a path of comparable points), and $\varphi$ be a monotone RDSI whose one-point motions are irreducible Markov chains. Then $\P(T < +\infty) =1$.
\et
\begin{proof}
Let $x$ be minimal in $S$, and $y_1 > x$. By irreducibility, given $s>0$
\[
\P(\varphi(s,\cdot)(y_1) = x) >0.
\]
Since $\varphi(s,\o) \in {\cal{M}}$ and ${\cal{M}}$ is finite, there exists $f_1 \in {\cal{M}}$ with $f_1(y_1) = x$ and
\[
\P(\varphi(s,\cdot) = f_1)>0.
\]
Note that, necessarily, $f_1(x) =x$, so $f_1$ is not bijective. Set $S_1 := f_1(S)$. Note that $S_1$, with the order induced by $S$, is connected, being the image of a connected posed under an increasing function. Clearly $x \in S_1$ is still minimal in $S_1$, and $|S_1| < |S|$. Unless $|S_1| =1$, the same argument can be repeated. Take $y_2 \in S_1$, $y_2 > x$, and $f_2 \in {\cal{M}}$ such that $f_2(y_2) = x$ and $P(\varphi(s,\o) = f_2)>0$. Again $f_2(x) = x$, so that $|S_2| := |f(S_1)| < |S_1|$. After a finite number of similar steps, we obtain a finite family $f_1,f_2,\ldots,f_n \in {\cal{M}}$ such that $\P(\varphi(s,\o) = f_i)>0$ and
\be{const}
f_n \circ f_{n-1} \circ \cdots \circ f_1 \equiv x \ \mbox{ is constant}.
\ee
Now, for $k=1,2,\ldots,n$, consider the events
\[
\{\varphi(s,\theta_{-ks}\cdot) = f_{n-k+1} \}.
\]
Since $\theta_t$ is $\P$-preserving, all these events have nonzero probability and, by independence of the increments, they are all independent. Thus, by (\ref{const})
\[
0 < \P \left( \bigcap_{k=1}^n \{\varphi(s,\theta_{-ks}\cdot) =
f_{n-k+1} \} \right) \leq \P(\varphi(ns,\theta_{-ns} \cdot) =
\mbox{ const.}).
\]
Now, let $t := ns$ and for $N \geq 1$ consider the events
$\{\varphi(t,\theta_{-Nt}\cdot) = \mbox{ const.} \}$. Since they
are independent and with the same nonzero probability,
\[
\P\left( \bigcup_N \{\varphi(t,\theta_{-Nt}\cdot) = \mbox{ const.}
\} \right) =1.
\]
Observing that $\varphi(t,\theta_{-Nt}\cdot) = \mbox{ const.}$
implies $\varphi(Nt,\theta_{-Nt}\cdot) = \mbox{ const.}$, we
obtain
\[
\P\left( \bigcup_N \{\varphi(Nt,\theta_{-Nt}\cdot) = \mbox{
const.} \} \right) =1.
\]
from which $\P(T < +\infty) =1$ follows.

\end{proof}
We conclude this section by remarking that a Markov chain with generator $L$ can be realized by a monotone RDSI if and only if it is realizably monotone. Indeed, if such RDSI exists, then $(\varphi(t ,\cdot))_{t \geq 0}$ is a Markov chain on $S^S$ for which the conditions in Definition \ref{mon2} are satisfied. Conversely, once we have the representation in (\ref{2.3}), a RDSI with the desired properties is obtained as follows. For $f \in {\cal{M}}$, let $\P_f$ be the law of a Poisson process on $\R$ with intensity $\Lambda(f)$, and, on the appropriate product space whose elements are denoted by $\o = (\o_f)_{f \in {\cal{M}}}$, we define $\P := \otimes_{f \in {\cal{M}}} \P_f$. The map $\varphi$ is constructed pointwise in $\o$ via the updating rule: if $t \in \o_f$ and $\varphi(t^-,\o) = g$ then $\varphi(t,\o) = f \circ g$.

\section{Extremal generators of stochastically monotone Markov chains: the monotonicity equivalence for ``small'' posets}
\label{equ} As seen in~subsection~\ref{cones}, equivalence between
realizable monotonicity and stochastic monotonicity of
any Markov Chain on a poset~$S$ is equivalent to
\begin{equation}\label{egualita_coni} {\cal{G}}_{r.mon}={\cal{G}}_{mon}.
\end{equation}
In this section we study monotonicity equivalence for posets with small
cardinality.\\
First note that the cases $\sharp S=2$,  $\sharp S=3$ are obvious:
in these cases $S$ is acyclic. According to
Theorem~\ref{fill-machida}, there is equivalence for
discrete-time Markov chains and using the result of
Proposition~\ref{impl} the equivalence holds for continuous-time
Markov chains as well.

In order to further investigate the equality (\ref{egualita_coni})
we developed computer programs. The cone~${\cal{G}}_{mon}$ is
defined as intersection of half spaces
in~(\ref{Cone_mon_H-representation}) (so called
\emph{H-representation}). The cone~${\cal{G}}_{r.mon}$ is defined
by its extremal rays in~(\ref{Cone_compl_mon_V-representation})
(so called \emph{V-representation}). The software \emph{cdd+} (see
\cite{cdd}) is able to compute exactly one representation given
the other one. This is a \emph{C++} implementation for convex
polyhedron of the Double Description Method (see for
instance~\cite{FukudaProdon}). Finding the extremal rays of the
cone~${\cal{G}}_{mon}$ and the (minimal) set of inequalities
defining the
cone~${\cal{G}}_{r.mon}$, the inclusion ${\cal{G}}_{mon}\subseteq  {\cal{G}}_{r.mon}$ can be easily checked.

We operated by first using the software \emph{GAP}
(see~\cite{GAP}) in order to
\begin{enumerate}[i.)]
\item find the up-sets $\Gamma$ related to the poset $S$, the
vectors $W^{\Gamma,x,y} \in \R^{S_2}$ and then identify the
\mbox{H-representation} of ${\cal{G}}_{mon}$;
\item \label{det_f_crescenti} compute all the increasing
functions~$f \in {\cal{M}}$, identify the vectors~$\mathbb{I}_f
\in (\R^+)^{S_2}$  and then find the \mbox{V-representation}
of~${\cal{G}}_{r.mon}$.
\end{enumerate}
We then use the software~\emph{cdd+} to produce the other
representations of the cones, and the software~\emph{Scilab}
to test if ${\cal{G}}_{mon}\subseteq
{\cal{G}}_{r.mon}$.

The difficulty in applying this method to posets with  high
cardinality is mainly due to the combinatorial complexity of the
step~(\ref{det_f_crescenti}) and to the computational time needed to \emph{cdd+} to obtain the dual representation of the cone. Rather than to $\sharp S$, this time
is related to the number of facets
of the cones, which comes from the partial order structure. It should also be remarked that a systematic
analysis, made by generating all posets with a given cardinality, is not doable for ``moderate'' cardinality.
For instance, the number of different unlabeled posets structure -- up to an order preserving isomorphism, not necessarily connected -- for a given set of cardinality~$16$ is $\sim 4.48 \times 10^{15}$. It was stated in 2002, see~\cite{Brinkmann}. For a set of cardinality~$17$, the number of unlabeled posets is till now unknown. For a set of cardinality~$4$, resp. $5$, $6$, $7$, the number of posets is respectively $16$, $63$, $318$, $2045$. See~\cite{unlabelledposets} for the list.

Nevertheless, we were able to completely study the cases when
$\sharp S\leq 6$. For $\sharp S > 6$, the
result of Proposition~\ref{extension} in the next section
gives the answer for some posets.\\

For $\sharp S=4$ the two relevant poset-structure are the diamond
and the bowtie.
Their Hasse-Diagram are given by the Figure~\ref{4posets}.
For those two posets, the algorithm above ensures that
${\cal{G}}_{mon}={\cal{G}}_{r.mon}$  holds. Note that this result
is known to be \emph{false} in discrete-time,
see for instance examples~$1.1$ and~$4.5$ in \cite{FM}.\\

Then, we studied all five-points posets which are not linearly
totally ordered. For some of these posets (see Figure~\ref{5posets} below), we
found extremal rays $L=(L_{x,y})_{(x,y)\in S_2}$
of~${\cal{G}}_{mon}$ which are not in~${\cal{G}}_{r.mon}$. One
example for each poset will be given below.

In what follows, a {\em symmetry} of a poset $S$ is a bijective map from $S$ to $S$ which is either order preserving or order reversing.
\begin{prop}\label{Prop3.1}
{The} only posets~$S$ with $\sharp S\leq 5$ such that
(\ref{egualita_coni}) does not hold are, up to symmetries, those whose Hasse-Diagrams
are presented in Figure~\ref{5posets}.

\begin{figure}[h]
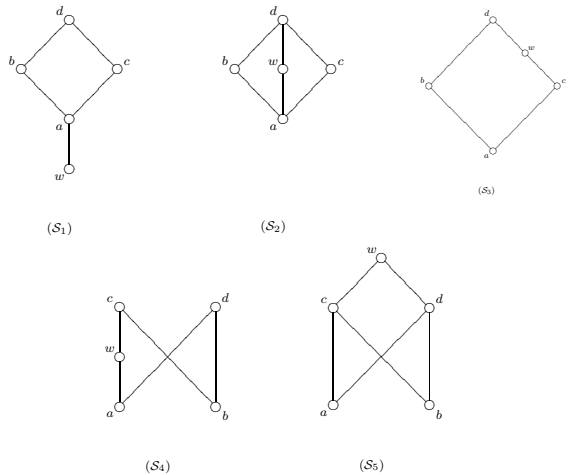

\centerline{ \scalebox{0.6}{\usebox{\Hasseprimo} %
\usebox{\Hassesecondo}} %
\scalebox{0.4}{\usebox{\Hassetertio}}}

\centerline{ \scalebox{0.6}{\usebox{\Hassequarto} %
\usebox{\Hassequinto} }}
\caption{\label{5posets}The five-points posets, for which there is no equivalence between the two notions of monotonicity in continuous-time.}
\end{figure}
\end{prop}
As mentioned above, Proposition \ref{Prop3.1} has been obtained by exact computer-aided
computations. For completeness, for each one of the posets in Figure \ref{5posets}, we give explicitly a generator in ${\cal{G}}_{mon} \setminus {\cal{G}}_{r.mon}$.
We recall
that, by Proposition~\ref{prop2.1}, a generator $L$ is realizably
monotone if and only if, for some
$\Lambda:\mathcal{M}\rightarrow\R^+$, we have $x,y\in S,\ \  x\neq
y \Rightarrow L_{x,y}=\sum_{f\in\mathcal{M}:f(x)=y}\Lambda(f)$. In
particular $L_{x,y}=0\Leftrightarrow \Lambda(f)=0$ \mbox{ for all }
\mbox{$f\in \mathcal{M}$} \mbox{ such that } $f(x)=y$. Given a function
$\Lambda$ as above, we shall use the abbreviate notation
$\mathcal{M}_{x\mapsto y}$ for the set $\{f\in \mathcal{M}:
\Lambda(f)>0 \mbox{ and } f(x)=y\}$.

\medskip

\beex \label{ex1}
For $S=S_1$ there is only one
extremal ray $L$ of~${\cal{G}}_{mon}$ which is not
in~${\cal{G}}_{r.mon}$. It is given by
$L_{d,c}=L_{d,b}=L_{b,w}=L_{c,w}=L_{a,w}=1$, and $L_{x,y}=0$ for each
other pair $x,y\in S_1$ with $x\neq y$. This generator is clearly
monotone (the conditions of Proposition~\ref{prop2.2} are easily
verified), but it is not realizably monotone: indeed, if
Proposition~\ref{prop2.1} holds, we have $L_{d,b} = \sum_{f \in
\mathcal{M}_{d\mapsto b}} \Lambda(f)=1$ and $L_{d,c} = \sum_{f \in
\mathcal{M}_{d\mapsto c}} \Lambda(f)=1$. Note that
$\mathcal{M}_{d\mapsto b}\cap\mathcal{M}_{d\mapsto c}=\emptyset$.
Moreover, for each $f\in \mathcal{M}_{d\mapsto b}$, since $c<d$
and $\Lambda(f)> 0$, by monotonicity of $f$ and the fact that
$L_{c,a}=L_{c,b}=0$, we have necessarily $f(c)=w$ and then $f(a)=w$ to maintain the ordering,
{\it i.e.}, $f\in \mathcal{M}_{a\mapsto w}$. Analogously,
$\mathcal{M}_{d\mapsto c}\subset \mathcal{M}_{a\mapsto w}$, then
$1=L_{a,w}=\sum_{f\in \mathcal{M}_{a\mapsto w}}\Lambda(f)\geq
\sum_{f\in \mathcal{M}_{d\mapsto b}\sqcup\mathcal{M}_{d\mapsto
c}}\Lambda(f)=2$, so we obtain a contradiction.
\eex \medskip

\beex \label{ex2}
For $S=S_2$ a generator
$L\in{\cal{G}}_{mon}\setminus {\cal{G}}_{r.mon}$ is given by
$L_{a,c}=L_{w,c}=L_{d,c}=L_{b,d}=L_{b,a}=1$ and $L_{x,y}=0$ for each
other pair $x,y\in S_2$ with $x\neq y$. According to Proposition~\ref{prop2.2} it is monotone.
Assume it is realizably
monotone. If $f$ is an increasing function on $S$ with
$f(d)=f(a)=c$, since $a<b<d$ we have $f(b)=c$. But $L_{b}=0$,
then necessarily $\Lambda(f)=0$. This means that
$\mathcal{M}_{d\mapsto c}\cap \mathcal{M}_{a\mapsto c}=\emptyset$.
Moreover, $(\mathcal{M}_{d\mapsto c}\sqcup \mathcal{M}_{a\mapsto
c})\subset \mathcal{M}_{w\mapsto c}$ which gives the contradiction
$1=L_{w,c}\geq L_{d,c}+L_{a,c}=2$.
\eex \medskip

\beex\label{ex3}
For $S=S_3$, consider the monotone
generator given by $L_{a,w}=L_{b,w}=L_{c,d}=L_{w,d}=L_{d,w}=1$ and
$L_{x,y}=0$ for each other pair $x,y\in S_3$ with $x\neq y$ and
suppose it is realizably monotone. If $f\in \mathcal{M}$ and
$f(d)=f(a)=w$, by monotonicity we have $f(c)=w$ and then, since
$L_{c,w}=0$, we have $\Lambda(f)=0$. Then $\mathcal{M}_{d\mapsto
w}\cap \mathcal{M}_{a\mapsto w}=\emptyset$ and by
$(\mathcal{M}_{d\mapsto w}\sqcup \mathcal{M}_{a\mapsto w})\subset
\mathcal{M}_{b\mapsto w}$, it follows that $L_{b,w}\geq 2$, which
gives a contradiction.
\eex \medskip

\beex \label{ex4} For $S=S_4$ we take the monotone generator
given by $L_{a,b}=L_{w,b}=L_{d,b}=L_{b,d}=L_{c,d}=1$ and $L_{x,y}=0$ for
each other pair $x,y\in S_4$ with $x\neq y$. It is clear that, if
$L$ was realizably monotone, we should have $\mathcal{M}_{d\mapsto
b}\cap\mathcal{M}_{b\mapsto d}=\emptyset$ and the inclusions
$\mathcal{M}_{d\mapsto b}\subset\mathcal{M}_{a\mapsto b}\subset
\mathcal{M}_{w\mapsto b}$ and $\mathcal{M}_{b\mapsto
d}\subset\mathcal{M}_{c\mapsto d}\subset \mathcal{M}_{w\mapsto
b}$: but it is not possible, since in that case we should have
$L_{w,b}\geq 2$.
\eex \medskip

\beex \label{ex5} For $S=S_5$ consider the monotone generator
given by $L_{c,a}=L_{d,a}=L_{b,a}=L_{w,c}=L_{w,d}=1$ and
$L_{x,y}=0$ for each other pair $x,y\in S_5$ with $x\neq y$.
If $L\in {\cal{G}}_{r.mon}$, we have $\mathcal{M}_{w\mapsto
d}\cap\mathcal{M}_{w\mapsto c}=\emptyset$, $\mathcal{M}_{w\mapsto
d}\subset\mathcal{M}_{c\mapsto a}\subset \mathcal{M}_{b\mapsto a}$
and $\mathcal{M}_{w\mapsto c}\subset\mathcal{M}_{d\mapsto
a}\subset \mathcal{M}_{b\mapsto a}$, then we obtain the
contradiction $1=L_{b,a}\geq 2$.
\eex \medskip

For the sake of completeness, for the posets considered in examples~\ref{ex1},...,
\ref{ex5} we give the number of extremal rays generating the cone ${\cal{G}}_{r.mon}$, resp. ${\cal{G}}_{mon}$
in $\mathbb R^{20}$, see Table~\ref{ConeSize5}.
\begin{table}[h!]
$$ \begin{array}{|c|c|c|}
\hline
\textrm{Poset }& {\cal{G}}_{r.mon} & {\cal{G}}_{mon} \\
\hline
\hline
S_1 & 40 & 41 \\
S_2 & 41 & 47 \\
S_3 & 40 & 42 \\
S_4 & 46 & 50 \\
S_5 & 49 & 53 \\
\hline
\end{array}$$
\caption{\label{ConeSize5}$\sharp S=5$: minimal number of extremal rays generating the cone ${\cal{G}}_{r.mon}$, resp. ${\cal{G}}_{mon}$, in~$\mathbb R^{20}$}
\end{table}

We now recall the following definition:
\begin{Def}\label{sub-poset} A subset~$S'$ of~$S$ is said to
be a \emph{sub-poset}
if for all $x,y\in S'$, $x\leq y$ in
$S'$ implies $x\leq y$ in $S$. $S'$ is said to be an
\emph{induced sub-poset} if for all $x,y\in S'$, $x\leq y$ in
$S'$ \emph{if and only if} $x\leq y$ in $S$.
\end{Def}
For $\sharp S=6$ we shall see in the next section that, if $S$ has
one of the 5-points posets above as an induced sub-poset, then
there is no equivalence between stochastic monotonicity and
realizable monotonicity. However, there are 6-points posets for
which there is no equivalence and such that we have equivalence
for each one of their sub-posets.
\begin{prop}\label{Prop3.2}
{The} only posets~$S$ with $\sharp S = 6$ such that
(\ref{egualita_coni}) does not hold are, up to symmetries,
\begin{itemize}
\item
those having one of the posets in Proposition \ref{Prop3.1} as induced sub-poset;
\item
those whose Hasse-Diagrams
are presented in Figure~\ref{6posets}.
\end{itemize}

\end{prop}
Following the terminology of \cite{FM}, the
first poset in Figure \ref{6posets} is a \emph{double diamond} and the second poset is a \emph{3-crown}.
\begin{figure}[h]
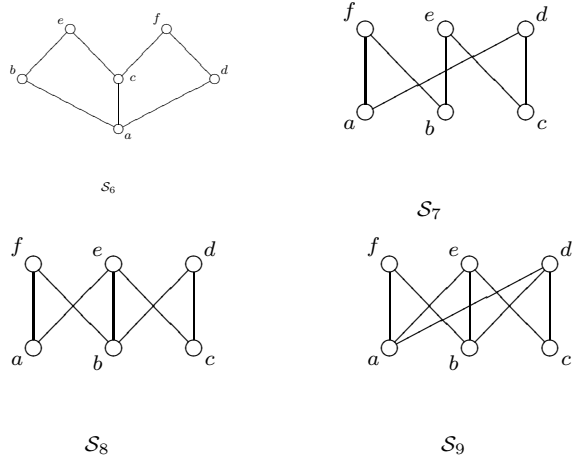

\centerline{ \scalebox{0.6}{\usebox{\HassePapillon}}
\usebox{\HasseCrown} %
}
\centerline{\usebox{\HasseDoubleBowTie} \usebox{\HasseCrownVariation}  
 }
\caption{\label{6posets}The double diamond $S_6$ and 3-crowns: $S_7, S_8,
S_9$.}
\end{figure}

In the next examples we shall see generators of
${\cal{G}}_{mon}\setminus {\cal{G}}_{r.mon}$ for posets $S_6, S_7,
S_8$ and $S_9$.
For the posets considered in examples~\ref{ex6},... \ref{ex9}
we give the number of extremal rays generating the cone ${\cal{G}}_{r.mon}$, resp. ${\cal{G}}_{mon}$
in $\mathbb R^{30}$ in Table~\ref{ConeSize6}.
\begin{table}[h!]
$$ \begin{array}{|c|c|c|}
\hline
Poset & {\cal{G}}_{r.mon} & {\cal{G}}_{mon} \\
\hline
\hline
S_6 & 126 & 421 \\
S_7 & 684 & 914 \\
S_8 & 134 & 312 \\
S_9 & 84 & 132 \\
\hline
\end{array}$$
\caption{\label{ConeSize6}$\sharp S=6$: number of extremal rays generating the cone ${\cal{G}}_{r.mon}$, resp. ${\cal{G}}_{mon}$
in $\mathbb R^{30}$}
\end{table}
\medskip

\beex \label{ex6}
Consider the monotone generator $L$ on $S_6$
defined as follows: $L_{a,c}=L_{d,c}=L_{c,b}=L_{b,e}=L_{f,e}=1$, and
$L_{x,y}=0$ for each other pair $x,y\in S_6$ with $x\neq y$. If $L$
was realizably monotone we should have $\mathcal{M}_{a\mapsto
c}\cap\mathcal{M}_{c\mapsto b}=\emptyset$, $\mathcal{M}_{a\mapsto
c}\subset\mathcal{M}_{d\mapsto c}$ and $\mathcal{M}_{c\mapsto
b}\subset \mathcal{M}_{f\mapsto e}\subset\mathcal{M}_{d\mapsto
c}$; but this would give the contradiction $1=L_{d,c}\geq 2$.
\eex \medskip

\beex \label{ex7} For $S=S_7$ we take $L\in {\cal{G}}_{mon}$
with $L_{a,c}=L_{b,c}=L_{f,c}=L_{d,c}=1$, $L_{c,d}=L_{e,d}=1$ and
$L_{x,y}=0$ for each other pair $x,y$ with $x\neq y$. Suppose $L$
is realizably monotone; then $\mathcal{M}_{d\mapsto
c}\cap\mathcal{M}_{c\mapsto d}=\emptyset$, $\mathcal{M}_{d\mapsto
c}\subset \mathcal{M}_{a\mapsto c}$ and $ \mathcal{M}_{c\mapsto
d}\subset \mathcal{M}_{e\mapsto d} \subset\mathcal{M}_{b\mapsto
c}\subset \mathcal{M}_{f\mapsto c}\subset\mathcal{M}_{a\mapsto
c}$. But in this case we have $1=L_{a,c}\geq L_{d,c}+L_{c,d}=2$, then
$L\notin {\cal{G}}_{r.mon}$.
\eex \medskip

\beex \label{ex8} For $S=S_8$ we consider the same generator of
Example 7; $\ L$ is clearly monotone, but realizable monotonicity
of $L$ would imply $\mathcal{M}_{d\mapsto c}\cap
\mathcal{M}_{c\mapsto d}=\emptyset$, $\mathcal{M}_{d\mapsto
c}\subset \mathcal{M}_{b\mapsto c}\subset \mathcal{M}_{f\mapsto
c}\subset \mathcal{M}_{a\mapsto c}$ and $\mathcal{M}_{c\mapsto
d}\subset \mathcal{M}_{e\mapsto d}\subset\mathcal{M}_{a\mapsto
c}$, then $L_{ac}\geq 2$, which is not the case. \eex \medskip

\beex \label{ex9} For $S=S_9$ let $L$ be defined by $
L_{a,c}=L_{b,c}=L_{e,c}=1$, $L_{b,e}=L_{f,e}=L_{d,e}=1$,
$L_{a,d}=L_{f,d}=L_{e,d}=1$ and $L_{x,y}=0$ for each other pair $x,y$
with $x\neq y$. $L$ is a monotone generator. Suppose $L\in {\cal{G}}_{r.mon}$. Then,
inequalities $a<e$ and $a<f$ imply respectively
$\mathcal{M}_{a\mapsto d}\cap \mathcal{M}_{e\mapsto c}=\emptyset$
and $\mathcal{M}_{a\mapsto d}\subset \mathcal{M}_{f\mapsto d}$.
Note that we have also $\mathcal{M}_{e\mapsto
c}\subset\mathcal{M}_{f\mapsto d}$: indeed $\mathcal{M}_{e\mapsto
c}\subset (\mathcal{M}_{f\mapsto e}\cup\mathcal{M}_{f\mapsto d})$
and, since $\mathcal{M}_{b\mapsto e}\subset \mathcal{M}_{f\mapsto e}$
and $L_{b,e}=L_{f,e}=1$ we have necessarily
$\mathcal{M}_{b\mapsto e}= \mathcal{M}_{f\mapsto e}$ and so
$\mathcal{M}_{e\mapsto c}\cap \mathcal{M}_{b\mapsto e}=\mathcal{M}_{e\mapsto c}\cap\mathcal{M}_{f\mapsto e}=\emptyset$.
Therefore we obtain the contradiction $L_{f,d}\geq 2$.
\eex \medskip

\begin{remark}\label{remark3}
We recall that, in discrete-time, equivalence does not hold if the
graph corresponding to the Hasse diagram of the poset has a
subgraph which is a cycle (in the graph-theoretic sense). So, a
look at the figures above could suggest that in continuous-time, a
sufficient condition for the failure of (\ref{egualita_coni}) is
the presence of \emph{two} cycles in the Hasse diagram. The poset
in Figure 4 (the complete $3$-crown) gives a counterexample: it has more than two cycles,
but for this set we have \emph{equivalence} between the two
concepts of monotonicity. In fact, more generally, if we have a poset
$S=\{a_1,\ldots, a_n, b_1,\ldots, b_m\}$ with $a_i< b_j$ for
$i=1,\ldots,n$, $j=1,\ldots,m$, we can show that every monotone
generator on $S$ is realizably monotone. We use an argument
analogous to the one used in section~5 of~\cite{FM}. Let $L$ be a
monotone generator on $S$ and consider the poset
$\tilde{S}=\{a_1,\ldots, a_n,c, b_1,\ldots, b_m\}$ with
$a_i<c<b_j$ for $i=1,\ldots,n$, $j=1,\ldots,m$. This poset admits $S$
as induced sub-poset and it is acyclic. Now, we take a generator
$\tilde{L}$ on $\tilde{S}$ defined as follows:
$\tilde{L}_{x,y}=L_{x,y}$ if $x,y\in S$, $\tilde{L}_{x,c}=0$ for each
$x\in S$, $\tilde{L}_{c,a_i}=\sum_{h=1}^m L_{b_h, a_i}$ for
$i=1,\ldots,n$, $\tilde{L}_{c,b_j}=\sum_{k=1}^n L_{a_k ,b_j}$ for
$j=1,\ldots,m$. It is not hard to check that this generator is monotone,  and its
restriction to $S$ is given by $L$. Since $\tilde{S}$ is acyclic,
then $\tilde{L}$ is realizably monotone. But a monotone realization
of $\tilde{L}$ gives a monotone realization of $L$ too. Therefore,
$L$ is realizably monotone.

\begin{figure}[h]
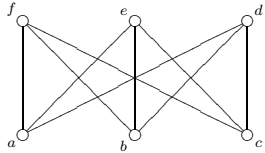

\centerline{\scalebox{0.7}{\usebox{\HasseCompleteCrown }}} \caption{\label{CompleteCrown}Complete $3$-crown}
\end{figure}
\end{remark}

\section{Extensions to larger posets}
\label{ext}
In this section we prove some sufficient conditions on a poset $S$
under which ${\cal{G}}_{mon}(S)\neq {\cal{G}}_{r.mon}(S)$. The
general argument we use is analogous to the one used in the
discrete-time case (see \cite{FM}, Theorem 4.2): we take a
generator $L$ in ${\cal{G}}_{mon}(S)\setminus
{\cal{G}}_{r.mon}(S)$ for a "small" poset $S$ and we define a
"monotone extension of $L$ to a larger poset $S'$", {\it i.e.}, a
generator $L'\in{\cal{G}}_{mon}(S')$, where $S'$ is a poset having
$S$ as induced sub-poset, such that $L'_{xy}=L_{xy}$ for all
$x,y\in S$. If this construction is possible, $L^\prime$ is not
realizably monotone too.
This is a consequence of the following Lemma
\begin{lemma}\label{m.extension}
Let $S$ be an induced sub-poset of a given poset $S'$ and let $L$
be a monotone generator on $S$ which has a monotone extension $L'$
to $S'$. Then $L'\in {\cal{G}}_{r.mon}(S')\Rightarrow L\in
{\cal{G}}_{r.mon}(S)$.
\end{lemma}
\begin{proof}
We denote by $\mathcal{M}'$ and $\mathcal{M}$ the sets of
increasing functions on $S'$ and $S$ respectively. Assume $L'\in
{\cal{G}}_{r.mon}(S')$. Then, by Proposition~\ref{prop2.1} there
exists $\bar{\Lambda}: \mathcal{M}'\rightarrow \R^+$ such that
(\ref{2.3}) holds for $L'$. Let us define $\Lambda:
\mathcal{M}\rightarrow \R^+$ by
$$\Lambda(f)=\sum_{f'\in\mathcal{M}', f'\mid_S=f}\bar{\Lambda}(f')$$
for each $f\in \mathcal{M}$. Now, since $L'_{x,y}=L_{x,y}\ $ for all
$x,y\in S$, in particular we have $L'_{x,x}=L_{x,x}$ for all $x\in S$ and then
necessarily $L'_{x,z}=0$ for all $x\in S, z\in
S'\setminus S$. Then, by condition (\ref{2.3}), for every $f'\in
\mathcal{M}'$ with $\bar{\Lambda}(f')>0$ we have $f'(S)\subset S$;
moreover, $S$ is an \emph{induced} sub-poset of $S'$, then we have
also $f'\mid_S\in\mathcal{M}$. Therefore, for $x,y\in S$, $x\neq
y$
\begin{eqnarray}
\nonumber \sum_{f\in\mathcal{M}:f(x)=y}\Lambda(f)&=& \sum_{f\in
\mathcal{M}:f(x)=y}\Big(\sum_{f'\in\mathcal{M}',
f'\mid_S=f}\bar{\Lambda}(f')\Big)\\
\nonumber&=&\sum_{\overset{f'\in\mathcal{M}', f'(x)=y}{
f'\mid_S\in\mathcal{M}}
}\bar{\Lambda}(f')\\
\nonumber&=&\sum_{f'\in\mathcal{M}'_{x\mapsto
y}}\bar{\Lambda}(f')=L'_{x,y}=L_{x,y},
\end{eqnarray}
then the proof is
complete using for $L$ Proposition~\ref{prop2.1}.
\end{proof}
As the Example below shows, Lemma \ref{m.extension} is false if
$S$ is a (not necessarily induced) sub-poset of $S'$, {\it i.e.} a
subset of $S'$ such that, for all $x, y\in S$, $\ x\leq y$ in $S$
implies $x\leq y$ in $S'$.
\medskip

\beex \label{ex10}
Let $S$ be the poset $S_8$ of Figure~\ref{6posets}: it is
a (not induced) sub-poset  of the complete crown of Figure 3, which
we denote by $S'$. Now let us consider the generator $L$ on $S$
defined by $L_{f,e}=L_{b,e}=L_{d,e}=1$, $L_{e,d}=L_{a,c}=L_{c,a}=1$ and
$L_{x,y}=0$ for each other pair $x,y\in S$ with $x\neq y$. It is
easy to check that $L$ is monotone as a generator both on $S$ and
on $S'$. Moreover, $L\notin {\cal{G}}_{r.mon}(S) $. Indeed, if $L$
was realizably monotone and $\mathcal{M}$ denotes the set of
increasing functions on $S$ we should have $\mathcal{M}_{e\mapsto
d}\subset\mathcal{M}_{a\mapsto c}\subset\mathcal{M}_{f\mapsto e}$,
$\mathcal{M}_{b\mapsto e}\subset\mathcal{M}_{f\mapsto e}$ and
$\mathcal{M}_{b\mapsto e}\cap\mathcal{M}_{e\mapsto d}=\emptyset$
which implies the contradiction $L_{f,e}\geq 2$. On the other hand,
as a generator on the complete crown, $L$ is a monotone extension
of itself and, by Remark~\ref{remark3},
$L\in{\cal{G}}_{r.mon}(S')$.
\eex \medskip

It must be stressed that the method of monotone extension of
generators to larger posets does not always work. First of all,
note that, if $L\in {\cal{G}}_{mon}(S)\setminus {\cal{G}}_{r.mon}(S)$
and there is  an \emph{acyclic} poset $S'$ which has $S$ as an
induced sub-poset, it is impossible to construct a monotone
extension of $L$ to $S'$: indeed, by Theorem~\ref{fill-machida}
and Proposition~\ref{impl} such an extension would be a generator
of $ {\cal{G}}_{r.mon}(S')$ and so, by Lemma~\ref{m.extension} we
should have also  $L\in {\cal{G}}_{r.mon}(S)$.\\
As an example, consider the poset $S_9$ of Figure~\ref{6posets} and the
generator $L$ of Example~\ref{ex9}. Consider the poset~$S_9'$ (see Figure~\ref{Aciclico}) obtained by
adding to $S_9$
the points $w, w_1, w_2$ in such a way that $a<w_1<w<w_2<d$,
$b<w_1<f$ and $c<w_2<e$: we obtain a 9-points poset which
is \emph{acyclic}. Therefore, it is impossible to obtain a
monotone extension of $L$ to $S_9^\prime$.\\
\begin{figure}[h]
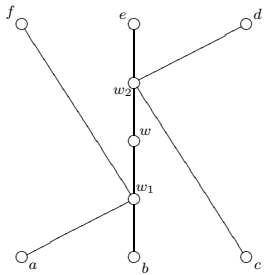

\centerline{\scalebox{0.7}{\usebox{\Aciclico}}} \caption{\label{Aciclico}An extension from poset $S_9$: acyclic poset $S_9'$}
\end{figure}

Moreover, the example below shows that, even when $S$ is not
"enlargeable" to an acyclic poset ({\it i.e.}, every poset having $S$ as
an induced sub-poset is \emph{non-acyclic}), there are generators
which cannot be extended to every "larger" poset.
\medskip

\beex \label{ex11} Let $L$ be the generator on poset~$S_8$
(Figure~\ref{6posets}) defined in Example~\ref{ex10}. Note that (see Proposition~5.11
of \cite{FM}) $S_8$ is not enlargeable to any acyclic poset.
However, consider a poset $S_8^\prime$ obtained from $S_8$ by adding a point $z$
with $z>f$, $z>e$ and $z\ngtr d$. Suppose $L^\prime$ is a
"monotone extension" of $L$. Since $\{d\}$ is a up-set, $L_{e,d} = 1$ and $z >e$, it must be $L^\prime_{z,d}\geq
1$. Moreover $\Gamma := \{d,b,c\}$ is a down-set;  since $L^\prime_{z,\Gamma} \geq 1$ and $z>f$, it should be $L_{f,\Gamma} \geq 1$, which is false, being $L_{f,\Gamma} =0$. Thus there can be no such extension. \\
However, as
we shall see in subsection~\ref{6points}, the method of monotone
extension
works for $S_8$, in the sense that for any poset $S^\prime$ having $S_8$ as induced sub-poset, \emph{there exists} a generator $L$ in $S_8$ which is monotone but not realizably monotone, that can be extended to a monotone  generator in $S'$. But this generator $L$ has to be chosen appropriately. For the reasons explained above, the extensions given in Sections \ref{sect4.1} and \ref{6points} have to be made "case by case".
\eex
\medskip

\subsection{From 5-points posets to larger posets}\label{sect4.1}
In this subsection we show that monotonicity equivalence does not
hold for any poset~$S'$ having one of the 5-points posets of
Figure~\ref{5posets} as an induced sub-poset. Note that, in this
case, $S'$ is \emph{non-acyclic}: this is an immediate consequence
of Proposition 2.7 of \cite{FM}.

\begin{prop}\label{extension} If a poset~$S'$ admits as induced sub-poset a poset~$S$,
whose Hasse-Diagram is one of those in Figure~\ref{5posets} (up to symmetries), then monotonicity
equivalence fails in~$S'$ as well.\end{prop}
\begin{proof}
For each $i=1,\ldots ,5$,  let $S_i^\prime$ be a poset which has
$S_i$ as an induced sub-poset. \\
According to Lemma~\ref{m.extension}, it is enough to show that, if we
choose a monotone generator $L$ on $S_i$ which is not realizably
monotone, then we can define a stochastically monotone generator
$L^\prime$ on $S_i^\prime$
such that $L^\prime_{x,y}=L_{x,y}$ for all $x,y\in S_i$.\\
In each case considered below we shall pose
$\bar{S}=S^\prime_i\setminus S_i$ and define $L^\prime$ in such a
way that the only new transitions allowed are the ones from
elements of $\bar{S}$ to elements of $S_i$. In other words, for
each $i=1,\ldots ,5$, we shall pose $L^\prime_{xy}=L_{xy}$ for all
$x,y\in S_i$ and
$L^\prime_{xy}=0$ for all $x\in S^\prime$ and $y\in \bar{S}$.\\
Note that, if $\Gamma^\prime$ is an up-set in $S_i^\prime$, then
$\Gamma=\Gamma^\prime\cap S_i$ is an up-set in $S_i$ and by the
construction of $L^\prime$ it follows that, for each $x\notin
\Gamma^\prime$, $L^\prime_{x,\Gamma^\prime}=L^\prime_{x,\Gamma}$.
The same property holds for down-sets.\\
Then (see Remark~\ref{remark1}), in order to verify that
$L^\prime$ is monotone, it will be sufficient to check that, for
all $x,y\in S_i^\prime$, with $x<y$, if $\Gamma$ is an up-set in
$S_i$ (resp. a down-set) and $x,y\notin \Gamma$, we have
$L^\prime_{x,\Gamma}\leq L^\prime_{y,\Gamma}$
(resp. $L^\prime_{x,\Gamma}\geq L^\prime_{y,\Gamma}$).\\

\emph{Case I.}  Let us consider $S_1$ and the generator $L$ given
in Example~\ref{ex1}.\\
We have to define only transition rates from elements of $\bar{S}$
to elements of $S_1$. Consider the partition of $S^\prime_1$ given
by the sets $A=\{z\in S^\prime_1: z\leq b\}\cup\{z\in S^\prime_1:
z\leq c\}$ and $B=S^\prime_1\setminus A$. Then, if $z\in
\bar{S}\cap A$ we pose $L^\prime_{zw}=1$ and $L^\prime_{zy}=0$ for
each other $y\in S_1$; if $z\in \bar{S}\cap B$, we pose $L^\prime_
{zb}=L^\prime_{zc}=1$ and $L^\prime_{zy}=0$ for each other $y\in
S_1$.\\
$L^\prime$ is a monotone generator. Indeed, let us take $x,y\in
S^\prime_1$ with $x<y$. There are only three possibilities:
$x,y\in B$, $x,y\in A$ or $x\in A$ and $y\in B$. If $x,y\in B$,
then $L^{\prime}_{xz}=L^{\prime}_{yz}$ for each $z\neq x,y$ and
there is nothing to verify. Now we show that if $\Gamma$ is an
up-set in $S_1$ such that $x,y\notin \Gamma$ and
$L^\prime_{x,\Gamma}\neq 0$ then $L^\prime_{x,\Gamma}\leq
L^\prime_{y,\Gamma}$. Suppose that $x,y\in A$; since, for $x\in A$
we have $L^\prime_{x,\Gamma}\neq 0$ if and only if $\Gamma=S_1$,
then necessarily $x,y\in \bar{S}\cap A$ and
$L^\prime_{x,S_1}=L^\prime_{y,S_1}=1$. For the same reason, if
$x\in A$ and $y\in B$, then $\Gamma=S_1$ and
$1=L_{x,S_1}^\prime<L^\prime_{y,S_1}=2$.

Analogously, consider a down-set $\Gamma$ in $S_1$ with $x,y\notin
\Gamma$ and $L^\prime_{y,\Gamma}\neq 0$. If $x,y\in A$, since each
down-set in $S_1$ contains $w$, we have $x,y\neq w$ and so
$L^\prime_{x,\Gamma}=L^\prime_{y,\Gamma}$. If $x\in A$ and $y\in
B$, then $\Gamma=\{b,a,w\}$ or $\{c,a,w\}$. Indeed, $y\in B$ and
$L^\prime_{y,\Gamma}\neq 0$, implies that
$\Gamma\cap\{b,c\}\neq\emptyset$. But
$x$ is smaller or equal than at least one element of the set
$\{b,c\}$ and $x\notin\Gamma$, so we cannot have
$\{b,c\}\subset\Gamma$. Therefore
$L^\prime_{x,\Gamma}=1=L^\prime_{y,\Gamma}$.
\medskip

\emph{Case II.}  Consider the partition of $S_2^\prime$ given by the sets $D=\{z\in
S_2^\prime: z\leq b\}\cup \{z\in S_2^\prime: z\leq w\}\cup\{z\in
S_2^\prime: z\leq c\}$ and $D^c = S_2^\prime \setminus D$. We pose,  for $z\in \bar{S}\cap
D$, $L^\prime_{z,a}=L^\prime_{z,c}=1$ and $L^\prime_{z,y}=0$ for
each other $y\in S_2$. If $z\in \bar{S}\cap D^c$, we pose
$L^\prime_{z,c}=L^\prime_{z,d}=1$ and $L^\prime_{z,y}=0$ for each
other $y\in S_2$.\\
Let us take $x,y\in S^\prime_2$ with $x<y$ and let $\Gamma$ be an
up-set in $S_2$ such that $x,y\notin \Gamma$ and
$L^\prime_{x,\Gamma}\neq 0$. Suppose first that $x,y\in D$. Since
$L^\prime_{x,\Gamma}\neq 0$ we must have $c\in \Gamma$ which
implies $\{c,d\}\subset \Gamma$. Now, if $x=a$ and $y\notin
\{b,w\}$, we have $\Gamma\neq S_2$ and
$L_{x,\Gamma}=1=L_{y,\Gamma}$; if $x\neq a$ and $y\in \{b,w\}$ we
have $\{a,b,w\}\notin \Gamma$ and so
$L_{x,\Gamma}=1=L_{y,\Gamma}$. Now, suppose $x,y\in D^c$. Since
$L^\prime_{x,\Gamma}\neq 0$ we must have $d\in \Gamma$ and
$x,y\neq d$: this means that $x,y\in \bar{S}\cap D^c$ and so
$L_{x,\Gamma}=L_{y,\Gamma}$. Finally, suppose that $x$ and $y$ are
not in the same subset of the given partition. This means that
$x\in D$ and $y\in D^c$. As we saw above $L^\prime_{x,\Gamma}\neq
0\Rightarrow \{c,d\}\subset \Gamma$, then necessarily $y\neq d$
and $L_{y,\Gamma}=2\geq L_{x,\Gamma}$. In order to check
monotonicity of $L^\prime$, we should consider also down-sets of
$S_2$, but in this case the argument for down-sets is symmetric to
the one given above for up-sets.

\medskip

\emph{Case III.} For the poset $S_3$ we consider the generator $L$
of Example~\ref{ex3} and take the partition of $S_3^\prime$ given by
$A=\{z\in S^\prime_3: z\leq b\}\cup\{z\in S^\prime_3: z< c\}$ and
$B=S^\prime_3\setminus A$. Then, if $z\in \bar{S}\cap A$ we pose
$L^\prime_{z,w}=1$ and $L^\prime_{z,y}=0$ for each other $y\in S_3$;
if $z\in \bar{S}\cap B$, we pose $L^\prime_ {z,w}=L^\prime_{z,d}=1$
and $L^\prime_{z,y}=0$ for each other $y\in S_3$.\\

In order to verify monotonicity of $L^\prime$, take $x,y\in
S_3^\prime$ with $x<y$. Note that $y\in A\Rightarrow x\in A$,
therefore we can have $x,y\in A$, $x,y\in B$ or $x\in A, y\in B$.
If $x,y\in A$, then $L_{x,z}^\prime=L^\prime_{y,z}$ for each $z\neq
x,y$ and there is nothing to check. Suppose that $x\in A, y\in
B$. Let us take an up-set $\Gamma$ in $S_3$ with $x,y\notin \Gamma$
and $L^\prime_{x,\Gamma}\neq 0$; then, $\{w,d\}\subset \Gamma$ and
$L^\prime_{x,\Gamma}=1\leq L^\prime_{y,\Gamma}$.
If we take a
down-set $\Gamma$ in $S_3$ with $x,y\notin \Gamma$ and
$L^\prime_{y,\Gamma}\neq 0$, since $x\in A\Rightarrow x<d$, we
have $d\notin \Gamma$, therefore $L^\prime_{y,\Gamma}=1=L^\prime_{x,\Gamma}$.
Now, suppose that $x,y\in B$. Let
us take an up-set $\Gamma$ with $x,y\notin \Gamma$ and
$L^\prime_{x,\Gamma}\neq 0$. Then, $x,y\neq d$ and, if
$w\notin\Gamma$, $L^\prime_{x,\Gamma}=L^\prime_{y,\Gamma}$. If
$\{w,d\}\subset \Gamma$ we have $x\neq w,d$ and (since
$y=c\Rightarrow x\in A$) $y\neq w,d,c$, therefore
$L^\prime_{y,\Gamma}=2\geq L^\prime_{x,\Gamma}$. If $\Gamma$ is a
down-set with $x,y\notin \Gamma$ and $L^\prime_{y,\Gamma}\neq 0$,
then $\Gamma=S_3$ or $\{w,a,c\}\subset\Gamma$. If $\Gamma=S_3$,
then $L_{x,\Gamma}^\prime=L_{y,\Gamma}^\prime=2$. If
$\{w,a,c\}\subset\Gamma$ we have $x,y\neq w,a,c$ and so
$L_{x,\Gamma}^\prime=L_{y,\Gamma}^\prime=1$.
\medskip

\emph{Case IV.}  For $S_4$ take the generator $L$ given in Example~\ref{ex4}.
Then we consider the partition of $S_4^\prime$ given by
$A=\{z\in S^\prime_4: b\leq z\leq d\} $, $B=\{z\in S^\prime_4:
z\geq b\}\setminus A$ and $C=S^\prime_4\setminus (A\cup B)$. Then,
if $z\in \bar{S}\cap A$ we pose $L^\prime_ {z,b}=L^\prime_{z,d}=1$
and $L^\prime_{z,y}=0$ for each other $y\in S_4$; if $z\in
\bar{S}\cap B$, we pose $L^\prime_{z,d}=1$ and $L^\prime_{z,y}=0$
for each other $y\in S_4$; if $z\in C$ we pose $L^\prime_{z,b}=1$
and $L^\prime_{z,y}=0$ for each other $y\in
S_4$.\\
Now we take $x,y\in S^\prime_4$ with $x<y$. Suppose that $x\in A$;
then $y>x\Rightarrow y\in A\cup B$. If $\Gamma$ is an up-set in
$S_4$ with $x,y\notin \Gamma$ and $L^\prime_{x,\Gamma}\neq 0$,
then $x, y\neq d$ and $x\notin\Gamma\Rightarrow b\notin\Gamma$,
therefore $L^\prime_{x,\Gamma}=1=L^\prime_{y,\Gamma}$. If $\Gamma$
is a down-set in $S_4$ with $x,y\notin \Gamma$ and
$L^\prime_{y,\Gamma}\neq 0$, note that, $d\in \Gamma\Rightarrow
x\in \Gamma$, therefore $\Gamma\cap\{b,d\}=\{b\}$ and so
$L^\prime_{x,\Gamma}=1=L^\prime_{y,\Gamma}$. Suppose now that
$x\in B$; since $y>x$ we have also $y\in B$ and there is nothing
to verify. Finally, for $x\in C$, if $\Gamma$ is an up-set with
$x,y\notin \Gamma$ and $L^\prime_{x,\Gamma}\neq 0$ then $b\in
\Gamma$, and so $\{b,d\}\subset\Gamma$, which implies that $y\ngeq
b$, {\it i.e.} $y\in C$ and so $L^\prime_{x,\Gamma}=
L^\prime_{y,\Gamma}$. If $\Gamma$ is a down-set with $x,y\notin
\Gamma$ and $L^\prime_{y,\Gamma}\neq 0$, then $b\in \Gamma$, $y\in
A\cup C$. Moreover, if $y\in A$ we have $d\notin \Gamma$,
therefore $L^\prime_{x,\Gamma}=1=L^\prime_{y,\Gamma}$.
\medskip

\emph{Case V.}  Consider the poset $S_5$ and the generator $L$ of
Example~\ref{ex5}. We take the partition of $S_5^\prime$ given by
$A=\{z\in S^\prime_5: z\leq c\}\cup\{ z\in S^\prime_5: z\leq d\}
$, $B=S^\prime_5\setminus A$. Then, if $z\in \bar{S}\cap A$ we
pose $L^\prime_ {z,a}=1$ and $L^\prime_{z,y}=0$ for each other $y\in
S_5$; if $z\in \bar{S}\cap B$, we pose $L^\prime_{z,,c}=
L^\prime_{z,d}=1$ and $L^\prime_{z,y}=0$ for each other $y\in S_5$.\\
Now we consider $x,y\in S^\prime_5$ with $x<y$. If both $x,y$
belong to $A$ or $B$, then $L^\prime_{x,z}=L^\prime_{y,z}$ for each
$z\neq x,y$, therefore there is nothing to verify. If $x,y$ are
not in the same set of the partition, since $y\in A\Rightarrow
x\in A$, we can have only $x\in A$ and $y\in B$. Suppose that
$\Gamma$ is an up-set in $S_5$ with $x,y\notin \Gamma$ and
$L^\prime_{x,\Gamma}\neq 0$. Then $a\in \Gamma$ and so $c,d\in
\Gamma$, which implies that $L^\prime_{y,\Gamma}=2\geq
L^\prime_{x,\Gamma}$. If $\Gamma $ is a down-set in $S_5$ with
$x,y\notin \Gamma$ and $L^\prime_{y,\Gamma}\neq 0$, we have
$\{c,d\}\cap\Gamma\neq\emptyset$. Moreover, since  $x\leq c$ or
$x\leq d$ and $x\notin \Gamma$, we have $\Gamma=\{c,a,b\}$ or
$\Gamma=\{d,a,b\}$ and in both cases we have
$L^\prime_{x,\Gamma}=1=L^\prime_{y,\Gamma}$.
\end{proof}

\subsection{From 6-points posets to larger posets}\label{6points}

As we saw in the preceding subsection, for a poset of
cardinality~$6$ having one of the 5-points posets of
Figure~\ref{5posets} as an induced sub-poset, there is not
equivalence between stochastic monotonicity and realizable
monotonicity. Therefore, the only six-points posets for which we
have to construct monotone extensions are posets $S_6, \ S_7,\
S_8$ of Figure~\ref{6posets}.

\begin{prop}\label{extension2}
If a poset~$S'$ admits as induced sub-poset a poset~$S$, whose
Hasse-Diagram is one of the posets $S_6, \ S_7,\  S_8$ of Figure~\ref{6posets}  then
monotonicity equivalence fails in~$S$ as well.
\end{prop}
\begin{proof}
\emph{Case I.} As we did in the preceding section, we call
$S^\prime$ a poset which has $S_6$ (the double diamond of Figure~\ref{5posets})
as an induced sub-poset, $\bar{S}=S^\prime\setminus S_6$,
and we take the monotone generator $L$ on $S_6$ defined in Example~\ref{ex6}.
This generator is not realizably monotone.\\
Now, we want to define $L^\prime$ on $S^\prime$ as a monotone extension
of $L$. We consider the partition of $S^\prime$ given by the sets
$A=\{z\in S^\prime: z> b\}\cup\{z\in S^\prime: z> c\}\cup\{z\in
S^\prime: z>d\}$ and $B=S^\prime\setminus A$. Then, if $z\in
\bar{S}\cap A$ we pose $L^\prime_{z,e}=1$ and $L^\prime_{z,y}=0$ for
each other $y\in S^\prime$; if $z\in \bar{S}\cap B$, we pose
$L^\prime_ {z,a}=L^\prime_{z,c}=1$ and $L^\prime_{z,y}=0$ for each
other $y\in S^\prime$.\\
$L^\prime$ is a monotone generator. Indeed, let us suppose that
$x,y\in S^\prime$ with $x<y$ and at least one of them does not
belong to $S_6$. We have only three possibilities: $x,y\in A$,
$x,y\in B$ or $x\in B$ and $y\in A$. If $x,y\in A$, and $x,y\neq
e$, there is nothing to verify, since $x$ and $y$ make the same
transitions with the same rate. On the other hand, if $x=e$ or
$y=e$, then, for every up-set (down-set) $\Gamma$ with
$x,y\notin\Gamma$ we have $e\notin \Gamma$ and so
$L^\prime_{x,\Gamma}=L^\prime_{y,\Gamma}=0$.\\
If $x,y\in B$ we have to consider only the cases in which $x$ and
$y$ make different transitions, {\it i.e.} when $y\in \{a,b,c,d\}$ and
$x\neq a$ or when $y\notin \{a,b,c,d\}$ and $x=a$.\\
In the first case, if we take an up-set $\Gamma$ with $x,y\notin
\Gamma$ and $L^\prime_{x,\Gamma}\neq 0$, then, since $y\geq a$ we
have $a\notin\Gamma$ and $\{c,e\}\subset\Gamma$. This implies
$L^\prime_{x,\Gamma}=L^\prime_{y,\Gamma}=1$; if we take a down-set
$\Gamma$ with $x,y\notin \Gamma$, then $a\in \Gamma$ and
$L^\prime_{x,\Gamma}\geq 1\geq L^\prime_{y\Gamma}$. In the second
case we have $x=a$, then, for any up-set $\Gamma$ with $y,a\notin
\Gamma$ we have $L^\prime_{a,\Gamma}\neq 0\Rightarrow
c\in\Gamma\Rightarrow L^\prime_{a,\Gamma}=L^\prime_{y,\Gamma}=1$;
on the other hand, for each down-set $\Gamma$ in $S^\prime$ with
$L^\prime_{y, \Gamma}\neq 0$ we have $x=a\in \Gamma$ and there is
nothing to verify.\\
Now, suppose $x\in B$ and $y\in A$. If $\Gamma$ is an up-set with
$L^\prime_{x, \Gamma}\neq 0$ and $x,y\notin \Gamma$, since $y>a$,
we have $a\notin \Gamma$, so $L^\prime_{x,\Gamma}\leq 1$ and
$L^\prime_{x,\Gamma}=1$ if and only if $\{c,e\}\subset \Gamma$
which implies $L^\prime_{y,\Gamma}=1$; if $\Gamma$ is a down-set
with $L^\prime_{y,\Gamma}\neq 0$, we have necessarily $e\in
\Gamma$, which implies $a,c\in \Gamma$, then
$2=L^\prime_{x,\Gamma}\geq L^\prime_{y,\Gamma}$.
\medskip

\emph{Cases II, III.}  Now, consider the monotone generator $L$ on $S_7$ given in Example~\ref{ex7}.
Note that $L$ has the same property also as a generator
on the poset $S_8$. If in the proof which follows we
consider $S_8$ instead of $S_7$, we obtain the same result.\\
Let $S'_7$ a poset which has $S_7$ as induced sub-poset and
$\bar{S}=S'_7\setminus S_7$. We take the partition of $S_7'$ given
by $A=\{z\in S'_7: c\leq z\leq d\}$, $B=\{z\in S'_7: z\geq
c\}\setminus A$, $C=S'_7\setminus (A\cup B)$ and we define a
monotone extension $L'$ of $L$ as follows: if $z\in \bar{S}\cap A$
we pose $L^\prime_ {z,c}=L^\prime_{z,d}=1$ and $L^\prime_{z,y}=0$ for
each other $y\in S_7$; if $z\in \bar{S}\cap B$, we pose
$L^\prime_{z,d}=1$ and $L^\prime_{z,y}=0$ for each other $y\in S_7$;
if $z\in C$ we pose $L^\prime_{z,c}=1$ and $L^\prime_{z,y}=0$ for
each other $y\in S_7$.\\
Now we take $x,y\in S^\prime_7$ with $x<y$. Suppose that $x\in A$;
then $y>x\Rightarrow y\in A\cup B$. If $\Gamma$ is an up-set in
$S_7$ with $x,y\notin \Gamma$ and $L^\prime_{x,\Gamma}\neq 0$,
then $x, y\neq d$ and $x\notin\Gamma\Rightarrow c\notin\Gamma$,
therefore $L^\prime_{x,\Gamma}=1=L^\prime_{y,\Gamma}$. If $\Gamma$
is a down-set in $S_7$ with $x,y\notin \Gamma$ and
$L^\prime_{y,\Gamma}\neq 0$, note that, $d\in \Gamma\Rightarrow
x\in \Gamma$, therefore $\Gamma\cap\{c,d\}=\{c\}$ and so
$L^\prime_{x,\Gamma}=1=L^\prime_{y,\Gamma}$. Suppose that $x\in
B$; since $y>x$ we have also $y\in B$ and there is nothing to
verify. Finally, for $x\in C$, if $\Gamma$ is an up-set with
$x,y\notin \Gamma$ and $L^\prime_{x,\Gamma}\neq 0$ then $c\in
\Gamma$, and so $\{c,d\}\subset\Gamma$, which implies that $y\ngeq
c$, {\it i.e.} $y\in C$ and so $L^\prime_{x,\Gamma}=
L^\prime_{y,\Gamma}$. If $\Gamma$ is a down-set with $x,y\notin
\Gamma$ and $L^\prime_{y,\Gamma}\neq 0$, then $c\in \Gamma$ and
$L^\prime_{x,\Gamma}=1$. On the other hand, if $y\in A$, then
necessarily $d\notin
\Gamma$, therefore in any case we have $L^\prime_{x,\Gamma}=1\geq L^\prime_{y,\Gamma}$.
\end{proof}

\begin{remark}
The procedure used in \mbox{Case~II} can be applied also to show that monotonicity
equivalence fails for every poset which has a \emph{$k$-crown}
(see Figure~\ref{kCrown})
with $k\geq 3$ as induced sub-poset.

\begin{figure}[h]
\centerline{\scalebox{0.6}{\usebox{\HassekCrown}}} \caption{\label{kCrown}k-crown}
\end{figure}

Let $S$ be a
$k$-crown and $\{x_0,\ldots, x_k\}$, resp. $\{y_0,\ldots,y_k\}$, be the
sets of its minimal, resp. maximal, elements (with
$x_k<y_k$, $x_k<y_{k-1}, \ x_{k-1}<y_{k-1}, x_{k-1}<y_{k-2}\ldots
x_0<y_0, x_0<y_k$). The generator defined by $L_{x_k,y_k}= %
L_{y_{k-1}, y_k}=1$, $L_{x_i, x_k}=1$ for $i=0,\ldots,k-1$, $L_{y_k, x_k}=1$,
$L_{y_i, x_k}=1$ for $i=0,\ldots k-2$ and $L_{x,y}=0$ for each other pair $x,y\in S$ with $x\neq y$,
is monotone. Suppose $L$ is realizably monotone. Then
$\mathcal{M}_{x_k\mapsto y_k}\subset \mathcal{M}_{y_{k-1}\mapsto y_k}\subset\mathcal{M}_{x_{k-1}\mapsto x_k}$ and
$\mathcal{M}_{y_k\mapsto x_k}\subset\mathcal{M}_{x_0\mapsto x_k}\subset \mathcal{M}_{y_0\mapsto x_k}\subset\ldots
\subset\mathcal{M}_{x_{k-1}\mapsto x_k}$. Therefore, since
$\mathcal{M}_{x_k\mapsto y_k}\cap\mathcal{M}_{y_k\mapsto x_k}=\emptyset$ we obtain $L_{x_{k-1},x_k}\geq 2$, which gives
a contradiction.\\
Now, as we did for the 3-crown, we consider a
poset $S'$ which has a $k$-crown as induced sub-poset, we pose
$\bar{S}=S'\setminus S$ and, in order to construct a monotone
extension of the generator $L$ given above, we take the partition
of $S'$ given by $A=\{z\in S': x_k\leq z\leq y_k\}$, $B=\{z\in S':
z\geq x_k\}\setminus A$, $C=S'\setminus (A\cup B)$. For $z\in
\bar{S}\cap A$ we pose $L^\prime_ {z,x_k}=L^\prime_{z,y_k}=1$ and
$L^\prime_{z,w}=0$ for each other $w\in S$; if $z\in \bar{S}\cap
B$, we pose $L^\prime_{z,y_k}=1$ and $L^\prime_{z,w}=0$ for each
other $w\in S$; if $z\in C$ we pose $L^\prime_{z,x_k}=1$ and
$L^\prime_{z,w}=0$ for each other $w\in S$. The same arguments used
above for the 3-crown show that $L'$ is monotone.
\end{remark}

\section{Conclusions} \label{conclusions}

In this paper we have obtained partial results concerning the
relationship between stochastic monotonicity and realizable
monotonicity for continuous-time Markov chains on partially
ordered sets. We have provided sufficient conditions on the poset
for the monotonicity equivalence to hold or to fail, and given
complete classifications for posets of cardinality $\leq 6$.
Unlike what Fill and Machida have obtained in the discrete-time
case, we have not been able to find a characterization of posets
for which monotonicity equivalence holds, in terms of their Hasse
diagram. We remark, as the example in Figure \ref{Aciclico} shows,
that there are posets with an acyclic extension for which
monotonicity equivalence fails. Therefore, in general,
non-equivalence is not preserved by extending the poset.

For posets with no acyclic extensions, we believe the following fact holds true.

\medskip
\noindent
{\bf Conjecture}.
Let $S$ be a connected poset having no acyclic extension. Then monotonicity equivalence holds if and only if the following conditions hold:
\begin{enumerate}[i)]
\item
the Hasse diagram of $S$ has a unique cycle, which is a diamond;
\item
$S$ has no $Y$-shaped subposet (see Figure~\ref{y}) having at most one point in common with the cycle in point i) and there is no induced subposet of the types  from Figure~\ref{forbidden} (up to symmetries).
\end{enumerate}

\begin{figure}[h]
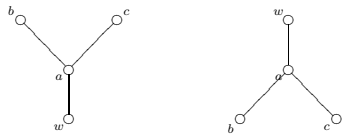

\centerline{ \scalebox{0.6}{
\usebox{\Hasseypsilon}
\usebox{\Hasseupsilon} %
}}
\caption{\label{y} Y shapes}
\end{figure}

\begin{figure}[h]
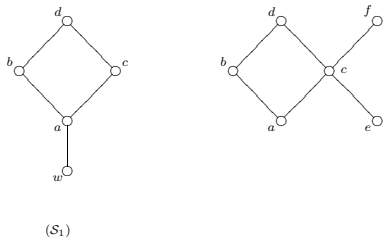

\centerline{ \scalebox{0.6}{
\usebox{\Hasseprimo}
\usebox{\HassePesce} %
}}
\caption{\label{forbidden} Forbidden posets}
\end{figure}

The necessity of conditions i) and ii) should actually be not too hard to prove, although many different cases have to be considered. We have tried harder to prove sufficiency of i), ii) by induction on the cardinality of the poset, but, unfortunately, we have not succeeded.

\section*{Acknowledgments}
The authors thanks the anonymous referees for useful comments and suggestions.
In particular, the relation with the weak monotonicity-equivalence has been suggested by one referee.

\providecommand{\bysame}{\leavevmode\hbox to3em{\hrulefill}\thinspace}
\providecommand{\MR}{\relax\ifhmode\unskip\space\fi MR }
\providecommand{\MRhref}[2]{%
 \href{http://www.ams.org/mathscinet-getitem?mr=#1}{#2}
}
\providecommand{\href}[2]{#2}

\end{document}